\documentclass[12pt]{amsart}
\usepackage{amsfonts}
\usepackage{amsmath}
\usepackage{amssymb}
\usepackage[mathcal]{eucal}
\usepackage{amscd}

\usepackage[pdftex,bookmarks,colorlinks,breaklinks]{hyperref}

\input xy
\xyoption{all}

\newcommand{\Mf}{\mathfrak{M}}

\newcommand{\R}{\mathbb{R}}

\newcommand{\N}{\mathbb{N}}

\newcommand{\Z}{\mathbb{Z}}

\newcommand{\XtoYGa}{$(X,\Ga)\overset\pi\to(Y,\Ga)$\ }

\newcommand{\al}{\alpha}
\newcommand{\Ga}{\Gamma}
\newcommand{\ga}{\gamma}
\newcommand{\del}{\delta}
\newcommand{\Del}{\Delta}
\newcommand{\ep}{\epsilon}
\newcommand{\sig}{\sigma}

\newcommand{\la}{\lambda}
\newcommand{\La}{\Lambda}
\newcommand{\tet}{\theta}

\newcommand{\om}{\omega}

\newcommand{\ka}{\kappa}

\newcommand{\XGa}{$(X,\Ga)\ $}

\newcommand{\OGa}{\mathcal{O}_{\Ga}}
\newcommand{\OCGa}{ {\ol{\mathcal{O}}}_{\Ga} }

\newcommand{\br}{\vspace{4 mm}}
\newcommand{\imp}{\ \Rightarrow\ }

\newcommand{\rest}{\upharpoonright}

\newcommand{\ol}{\overline}
\newcommand{\ch}{\mathbf{1}}

\newcommand{\cls}{\rm{cls\,}}

\newcommand{\Aut}{\rm{Aut\,}}

\newcommand{\gr}{\rm{graph\,}}

\newcommand{\diam}{\rm{diam\,}}
\newcommand{\supp}{\rm{supp\,}}

\newcommand{\card}{\rm{card\,}}

\newcommand{\Homeo}{\rm{Homeo\,}}

\swapnumbers
\theoremstyle{plain}
\newtheorem{thm}{Theorem}[section]
\newtheorem{cor}[thm]{Corollary}
\newtheorem{lem}[thm]{Lemma}
\newtheorem{prop}[thm]{Proposition}

\theoremstyle{definition}
\newtheorem{defn}[thm]{Definition}

\newtheorem{rmk}[thm]{Remark}

\newtheorem{prob}[thm]{Problem}

\numberwithin{equation}{section}

\begin{document}
\title
[Tame minimal dynamical systems]{The structure of tame minimal
dynamical systems for general groups}

\author[Eli Glasner]{Eli Glasner}
\address{Department of Mathematics,
Tel-Aviv University, Tel Aviv, Israel}
\email{glasner@math.tau.ac.il}

\date{June, 2016}

\begin{abstract}
We use the structure theory of minimal dynamical systems
to show that, for a general group $\Ga$,
a tame, metric, minimal dynamical system $(X, \Ga)$ has the following structure:
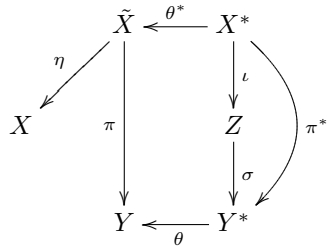
\begin{equation*}
\xymatrix
{
& \tilde{X} \ar[dd]_\pi  \ar[dl]_\eta & X^* \ar[l]_-{\tet^*} \ar[d]^{\iota} \ar@/^2pc/@{>}^{\pi^*}[dd]\\
X & & Z \ar[d]^\sig\\
& Y & Y^* \ar[l]^\tet
}
\end{equation*}
Here (i) $\tilde{X}$ is a metric minimal and tame system
(ii) $\eta$ is a strongly proximal extension, (iii) $Y$ is a strongly proximal system,
(iv) $\pi$ is a point distal and RIM extension with unique section, (v) $\tet$, $\tet^*$ and $\iota$
are almost one-to-one extensions, and (vi) $\sig$ is an isometric
extension. 
%
%

When the map $\pi$ is also open this diagram reduces to
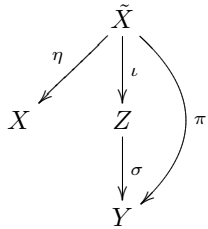
\begin{equation*}
\xymatrix
{
& \tilde{X}  \ar[dl]_\eta  \ar[d]^{\iota} \ar@/^2pc/@{>}^\pi[dd]\\
X &  Z \ar[d]^\sig\\
& Y
}
\end{equation*}

In general the presence of the strongly proximal extension $\eta$ is unavoidable.
If the system $(X, \Ga)$ admits an invariant measure $\mu$ then 
$Y$ is trivial and $X = \tilde{X}$ is an almost automorphic system;
i.e. $X \overset{\iota}{\to} Z$, where $\iota$ is an almost one-to-one 
extension and $Z$ is equicontinuous. 
Moreover, $\mu$ is unique and $\iota$ is a measure theoretical isomorphism
$\iota : (X,\mu, \Ga) \to (Z, \la, \Ga)$, with $\la$ the Haar measure on $Z$.
Thus, this is always the case when $\Ga$ is amenable.
\end{abstract}

\thanks{{\em 2000 Mathematical Subject Classification
54H20}}

\thanks{This research was supported by a grant of the Israel Science Foundation (ISF 668/13)}


\maketitle

\tableofcontents

\section*{Introduction}
In this work we prove a structure theorem for tame, metric, minimal dynamical systems 
for an arbitrary acting group, Theorem \ref{main} and Corollary \ref{cor}. 
This generalizes an older result which delt with
the case of abelian acting group (\cite{H}, \cite{KL} and \cite{Gmt}).
It turns out that the key property needed for the simple form of the structure
theorem is amenability, rather than commutativity. 
In particular Corollary \ref{cor} asserts that, when the acting group 
is amenable, every metric minimal tame system is almost automorphic.
When the group is
not amenable the structure of a minimal tame system is more complicated
and may involve a nontrivial strongly proximal factor.

The method of proof we use is basically the same one we applied in
\cite{Gmt} and consists mainly in applying various structure theorems
for general minimal systems. 
After some efforts to write our proof as a commentary on \cite{Gmt} with 
indications where new ideas and results appear, we found that
the best way (for the reader as well as for the author) to present the new results is to start anew, 
although for some minor points
we nevertheless refer to the older work.
For completeness' sake and in order to provide the reader with the appropriate
background I also reproduce the review section, with the
necessary changes.

\br

A dynamical system is a pair $(X,\Ga)$,
where $X$ is a compact Hausdorff space and $\Ga$ an infinite
abstract group acting as a group of homeomorphisms of the space $X$.
That is, we are given a homomorphism (not necessarily an
isomorphism) of $\Ga$ into $\Homeo(X)$.
For $\ga\in \Ga$ and $x\in X$ we write $\ga x$ for the image
of $x$ under the homeomorphism which corresponds to $\ga$.
We will often abuse this notation and consider $\ga$ as
a homeomorphism of $X$.

The {\em enveloping semigroup} $E(X,\Ga)$ of the dynamical system \XGa
is defined as the closure of image of $\Ga$ in the product space $X^X$.
It is not hard to check that, under composition of maps,
$E(X,\Ga)$ is a compact {\em right topological semigroup} , i.e.
for each $q\in E(X,\Ga)$ the map $R_q:p \mapsto pq$ is continuous.
In fact the canonical homomorphism of $\Ga$ into $E(X,\Ga)$ is a
{\em right topological semigroup compactification} of $\Ga$; i.e.
it has a dense range and for each $\ga \in \Ga$
multiplication on the left $L_\ga: p \mapsto \ga p$
is continuous on $E(X,\Ga)$. This left multiplication by elements of $\Ga$
makes $(E(X,\Ga),\Ga)$ a dynamical system.

The enveloping semigroup was introduced by Robert Ellis in 1960 and
became an indispensable tool in abstract topological
dynamics. However explicit computations of enveloping
semigroups are quite rare.
One reason for this is that often $E(X,\Gamma)$ is non-metrizable.

Following a pioneering work of A. K\"{o}ller, \cite{Ko},
Glasner and Megrelishvily proved the
following dynamical version of the
Bourgain-Fremlin-Talagrand dichotomy theorem, \cite{G}, \cite{GM}.

\begin{thm}[A dynamical BFT dichotomy]
Let $(X,\Gamma)$ be a metric dynamical system and let $E(X, \Gamma)$
be its enveloping semigroup. We have the following dichotomy.
Either
\begin{enumerate}
\item
$E(X, \Gamma)$ is a separable Fr\'{e}chet  compact space, hence with cardinality
${\card}\ {E(X, \Gamma)} \leq 2^{\aleph_0}$; or
\item
the compact space $E$ contains a homeomorphic
copy of $\beta\mathbb{N}$, hence ${\card}\ {E(X, \Gamma)} = 2^{2^{\aleph_0}}$.
\end{enumerate}
\end{thm}

A dynamical system is called {\em tame} if the first alternative occurs,
i.e. $E(X, \Gamma)$ is  Fr\'{e}chet.
A dynamical characterization of tame metrizable dynamical systems follows,
\cite{GM} and \cite{GMU}:

\begin{thm}
A compact metric dynamical system $(X, \Gamma)$ is tame if and only if
every element of $E(X, \Gamma)$ is a Baire class 1 function from $X$ to itself.
\end{thm}

\br

For the definitions of HAE (hereditarily almost equicontinuous)
systems and other undefined notions which appear in the following theorem,
as well as for some further motivation and examples we refer the reader
to the papers \cite{GM}, \cite{G} and \cite{GMU}.

\begin{thm}[\cite{GMU}]
\label{th:main}
Let $X$ be a compact metric $G$-space.
The following conditions are equivalent:
\begin{enumerate}
\item
the dynamical system $(X, \Gamma)$ is hereditarily almost equicontinuous
(HAE);
\item
the dynamical system $(X, \Gamma)$ is
RN, that is,
it admits a proper representation on a Radon--Nikod\'ym
Banach space;
\item
the enveloping semigroup $E(X, \Gamma)$ is metrizable.
\end{enumerate}
\end{thm}
It thus follows that every metric HAE system is tame.

\br

I would like to thank the two anonymous referees for their corrections and helpful suggestions.

\section{A brief survey of abstract topological dynamics}\label{sur-sec}
This section is a brief review of the structure theory of
minimal dynamical systems. We will emphasize some
aspects which will be relevant in the present work.
For full details the reader is referred to
the books \cite{E2}, \cite{Gl1}, \cite{Au} and \cite{Vr}
and the review articles \cite{V} and \cite{G-lu}.

If $X$ is a topological space and $A \subset X$ then we use 
$\ol{A} = \cls A$, $A^\circ$, $\partial A$ and $A^c = X \setminus A$ for the
closure, interior, boundary and complementation operations
on $A$, respectively.

A {\em topological dynamical system} or briefly a system
is a pair $(X,\Ga)$, where
$X$ is a compact Hausdorff space and $\Ga$ an abstract infinite group
which acts on $X$ as a group of homeomorphisms.
When there is no room for confusion we write $X$ for
the system $(X,\Ga)$.
A {\em sub-system} of \XGa is a closed invariant subset
$Y\subset X$ with the restricted action.
For a point $x\in X$, we let $\OGa (x)=\{\ga x:\ga \in \Ga\}$,
and $\OCGa (x)={\cls} \{\ga x:\ga \in \Ga\}$. These subsets of $X$
are called the {\em orbit} and
{\em orbit closure} of $x$ respectively.
We say that $(X,\Ga)$ is {\em point transitive} if
there exists a point $x\in X$ with a dense orbit.
In that case $x$ is called a {\em transitive point}.
If every point is transitive we say that $(X,\Ga)$ is a
{\em minimal system}. We say that $x\in X$ is an {\em
almost periodic} or a {\em minimal} point if $\OCGa(x)$
is a minimal system.

The dynamical system \XGa is {\em topologically transitive}
if for any two nonempty open subsets $U$ and $V$ of $X$
there exists some $\ga \in \Ga$ with $\ga U\cap V\ne\emptyset$.
Clearly a point transitive system is topologically transitive
and when $X$ is metrizable the converse holds as well:
in a metrizable topologically transitive system the
set of transitive points is a dense $G_\delta$ subset
of $X$.

The system \XGa is {\em weakly mixing} if the product
system $(X\times X, \Ga)$ (where $\ga(x,x')=(\ga x,\ga x'),
\ x,x'\in X,\ \ga \in \Ga$) is topologically transitive.

If $(Y,\Ga)$ is another system then a continuous onto
map $\pi:X\to Y$ satisfying $\ga\circ \pi=\pi\circ \ga$
for every $\ga \in \Ga$ is called a {\em homomorphism}
of dynamical systems. In this case we say that $(Y,\Ga)$
is a {\em factor} of $(X,\Ga)$ and also that
$(X,\Ga)$ is an {\em extension} of $(Y,\Ga)$.
With the system \XGa we associate the induced action
(the {\em hyper system} associated with \XGa)
on the compact space $2^X$ of closed subsets of $X$
equipped with the Vietoris topology.
A subsystem  $Y$ of $(2^X,\Ga)$ is a {\em
quasifactor} of \XGa if $\bigcup \{A: A\in Y\}=X$.

The system \XGa can always be considered as a quasifactor
of \XGa by identifying $x$ with $\{x\}$.
Recall that if \XtoYGa is a homomorphism then in general
$\pi^{-1}:Y\to 2^X$ is an upper-semi-continuous map
and that
$\pi:X\to Y$ is open iff $\pi^{-1}:Y\to 2^X$ is continuous.

\br

We assume for simplicity that our acting group $\Ga$
is a discrete group. $\beta \Ga$ will denote the
Stone-\v {C}ech compactification of $\Ga$.
The universal properties of $\beta \Ga$ make it
\begin{itemize}
\item
a compact semigroup with right continuous
multiplication (for a fixed $p\in \beta \Ga$
the map $q\mapsto qp,\ q\in \beta \Ga$ is continuous),
and left continuous multiplication by
elements of $\Ga$, considered as elements of $\beta \Ga$
(for a fixed $\ga\in \Ga$
the map $q\mapsto \ga q,\ q\in \beta \Ga$ is continuous).
\item
a dynamical system $(\beta \Ga, \Ga)$ under
left multiplication by elements of $\Ga$.
\end{itemize}

The system $(\beta \Ga, \Ga)$ is the universal point
transitive $\Ga$-system; i.e. for every point
transitive system \XGa and a point $x\in X$ with
dense orbit, there exists a homomorphism of
systems $(\beta \Ga, \Ga)\to (X,\Ga)$ which sends $e$,
the identity element of $\Ga$, onto $x$. For $p\in \beta \Ga$
we let $px$ denote the image of $p$ under this homomorphism.
This defines an ``action" of the semigroup $\beta \Ga$ on
every dynamical system. In fact, by universality there exists
a unique homomorphism $(\beta \Ga, \Ga)\to (E(X,\Ga),\Ga)$
onto the enveloping semigroup $E(X,\Ga)$ which is also a semigroup
homomorphism and we can interpret, and often do, the $\beta \Ga$
action on $X$ via this homomorphism.

When dealing with the hyper system $(2^X,\Ga)$ we write
$p\circ A$ for the image of the closed subset $A\subset X$
under $p\in \beta \Ga$ to distinguish it
from the (usually non-closed) subset $pA=\{px:x\in A\}$.
If $p$ is the limit of a net $\ga_i$ in
$\Ga$ then
$$
p \circ A =\{x\in X:
\text{there are a subnet $\ga_{i_j}$ and a net $x_j \in A$
with $x=\lim_j \ga_{i_j} x_j$}\}.
$$
We always have $pA\subset p\circ A$.

\br

The compact semigroup $\beta \Ga$ has a
rich algebraic structure.
For instance for countable $\Ga$
there are $2^c$ minimal
left (necessarily closed) ideals in
$\beta \Ga$ all isomorphic as systems
and each serving as a universal minimal system.
Each such minimal ideal, say $M$, has a subset $J$
of $2^c$ idempotents such that $\{vM:v\in J\}$ is a
partition of $M$ into disjoint isomorphic (non-closed)
subgroups. An idempotent in $\beta\Ga$ is called {\em minimal} if
it belongs to some minimal ideal.
A point $x$ in a dynamical system \XGa is
a minimal point iff there is some minimal idempotent $v$ in $\beta\Ga$
with $vx=x$, iff there exists some $v \in J$ with $vx=x$.

The group of dynamical system automorphisms of $(M,\Ga)$,\
$G={\Aut}(M,\Ga)$ can be identified with any one of the
groups $vM$ as follows:
with $\al\in vM$ we associate
the automorphism $\hat\al:(M,\Ga)\to (M,\Ga)$ given by
right multiplication $\hat\al(p)=p\al,\ p\in M$.
The group $G$ plays a central role in the algebraic theory.
It carries a natural $T_1$ compact topology, called by
Ellis the $\tau$-{\em topology}, which is weaker than the
relative topology induced on $G=uM$ as a subset of $M$.
The $\tau$-closure of a subset $A$ of $G$
consists of those $\beta\in G$ for which the set
${\gr}(\beta)=\{(p,p\beta):p\in M\}$
is a subset of the closure in $M\times M$
of the set $\bigcup\{{\gr}(\alpha):\alpha\in A\}$.
Both right and left multiplication on $G$ are $\tau$ continuous
and so is inversion.

It is convenient to fix a minimal left ideal $M$
in $\beta \Ga$ and an idempotent $u\in M$. As explained above we
identify $G$ with $uM$ and it follows that
for any subset $A\subset G$,
$$
{\cls}_\tau A = u(u\circ A) = G \cap (u\circ A).
$$
Also in this way we can consider
the ``action" of $G$ on every system \XGa via the
action of $\beta \Ga$ on $X$.
With every minimal
system \XGa and a point $x_0\in uX=\{x\in X: ux=x\}$ we
associate a $\tau$-closed subgroup
$$
\mathfrak{G}(X,x_0)=\{\alpha\in G:\alpha x_0=x_0\},
$$
the {\em Ellis group} of the pointed
system $(X,x_0)$. The quotient space $G/\mathfrak{G}(X,x_0)$
can be identified with the subset $uX \subset X$ via the map
$\al \mapsto \al x_0$
and the induced quotient $\tau$-topology is called the {\em
$\tau$-topology} on $uX$. Again the $\tau$-topology is weaker than
the relative topology induced on $uX$ as a subset of $X$, it is
$T_1$ and compact, and the closure operation is given by
$$
{\cls}_\tau A = u(u\circ A) = uX \cap (u\circ A),
\qquad A \subset uX.
$$

For a homomorphism $\pi: X \to Y$
with $\pi(x_0)=y_0$ we have
$$
\mathfrak{G}(X,x_0)\subset \mathfrak{G}(Y,y_0).
$$

For a $\tau$-closed subgroup $F$ of $G$ the {\em derived group}
$F'$ is given by:
\begin{equation}\label{prime}
F':=\bigcap\{{\cls}_\tau O : O\
\text{a $\tau$-open neighborhood of\ } u\  \text{in} \ F\}.
\end{equation}
$F'$ is a $\tau$-closed normal
(in fact characteristic) subgroup of $F$ and it is
characterized as the smallest $\tau$-closed subgroup
$H$ of $F$ such that $F/H$ is a compact
Hausdorff topological group.
In particular, for an abelian $\Ga$, the topological group
$G/G'$ is the Bohr compactification of $\Ga$.

\br

A pair of points $(x,x')\in X\times X$ for a system \XGa
is called {\em proximal} if there exists a net
$\ga_i\in \Ga$ and a point $z\in X$ such that
$\lim \gamma_ix=\lim \gamma_ix'=z$
(iff there exists $p\in\beta \Ga$ with $px=px'$).
We denote by $P$ the set of proximal pairs in
$X\times X$.
We have
$$
P=\bigcap\ \{\Ga V:V\
{\text{\rm a neighborhood of the diagonal in}}\ X\times X\}.
$$
We write $P[x] := \{x' \in X : (x,x') \in P\}$.
A system \XGa is called {\em proximal} when $P=X\times X$
and {\em distal} when $P=\Delta$, the diagonal in $X\times X$.
It is called {\em strongly proximal} when the following
much stronger condition holds: the dynamical system
$(\Mf(X),\Ga)$, induced on the compact space $\Mf(X)$ of Borel probability
measures on $X$, is proximal.
A minimal system \XGa is called {\em point distal}
if there exists a point $x\in X$ such that if
$x,x'$ is a proximal pair then $x=x'$.

The {\em regionally proximal relation\/} on
$X$ is defined by
$$
Q=\bigcap\ \{\ol {\Ga V}:V\
{\text{\rm a neighborhood of $\Delta$ in}}\ X\times X\}.
$$
It is easy to verify that $Q$ is trivial --- i.e. equals
$\Delta$ --- iff the system is equicontinuous.

\br

An extension \XtoYGa of minimal systems
is called a {\em proximal extension}
if the relation $R_\pi=\{(x,x'):\pi(x)=\pi(x')\}$
satisfies $R_\pi\subset P$, and
a {\em distal extension} when $R_\pi\cap P=\Del$.
One can show that every distal extension is open.
$\pi$ is a {\em highly proximal (HP) extension} if for every closed
subset $A$ of $X$ with $\pi(A)=Y$, necessarily $A=X$.
It is easy to see that a HP extension is proximal.
In the metric case an extension \XtoYGa of minimal systems is
HP iff it is an {\em almost  1-1 extension}, that is the set
$\{x\in X:   \pi^{-1}(\pi(x)) =\{x\}\}$
is a dense $G_\del$ subset of $X$.

The map $\pi$ is {\em strongly proximal} if for every $y\in Y$
and every probability measure $\nu$ with ${\supp}\nu \subset \pi^{-1}(y)$,
there exists a net $\ga_i\in \Ga$ and a point $x\in X$ such that
$\lim_i \ga_i \nu =\delta_x$ in the weak$^*$ topology on the space
$\Mf(X)$ of probability measures on $X$.

The extension $\pi$ is called an
{\em equicontinuous extension} if for every $\ep$,
a neighborhood of the diagonal
$\Delta = \{(x,x): x\in X\} \subset X\times X$,
there exists a neighborhood of the diagonal $\del$ such that
$\ga (\del\cap R_\pi) \subset \ep$ for every $\ga \in \Ga$.
In the metric case an equicontinuous extension is also called
an {\em isometric extension}.

A minimal dynamical system $(X, \Ga)$ is called 
{\em almost automorphic} if
it has the form $X \overset{\tet}{\to} Z$, where $Z$ is equicontinuous and $\tet$
is almost one-to-one. More generally, an extension of minimal systems
$\pi : X \to Y$ is called an {\em almost automorphic extension} if it has the form
$X \overset{\tet}{\to} Z \overset{\sig}{\to} Y$, where $\tet$ is an almost one-to-one
extension, $\sig$ is an isometric extension and $\pi = \sig \circ \tet$.

The extension $\pi$ is a {\em weakly mixing extension}
when $R_\pi$ as a subsystem of the product system
$(X\times X,\Ga)$ is topologically transitive.

\br

The algebraic language is particularly suitable for dealing
with such notions. For example an extension
\XtoYGa of minimal systems is a proximal extension iff
the Ellis groups
$\mathfrak{G} (X,x_0)=A$ and $\mathfrak{G} (Y,y_0)=F$ coincide.
It is distal iff for every $y\in Y$,
and $x\in \pi^{-1}(y),\
\pi^{-1}(y)=\mathfrak{G}(Y,y)x$;\
iff:
\begin{quote}
for every $y=py_0\in Y$, $p$ an element of $M$,
$\pi^{-1}(y)=p\pi^{-1}(y_0)=
pFx_0$, where $F=\mathfrak{G}(Y,y_0)$.
\end{quote}
In particular $(X,\Ga)$ is distal iff
$Gx=X$ for some (hence every) $x\in X$.
The extension $\pi$ is an equicontinuous extension
iff it is a distal extension and, denoting
$\mathfrak{G} (X,x_0)=A$ and $\mathfrak{G} (Y,y_0)=F$,
$$
F'\subset A.
$$
In this case, setting $A_0=\bigcap_{g\in F} gAg^{-1}$,
the group $F/A_0$ is the
group of the {\em group extension\/}\ $\tilde \pi$ associated with
the equicontinuous extension $\pi$.
More precisely, there exists a minimal
dynamical system $(\tilde X,\Ga)$, with
${\mathfrak{G}}(\tilde{X},\tilde{x}_0)=A_0$,
on which the compact Hausdorff topological
group $K=F/A_0$ acts as a group of automorphisms and
we have the following commutative diagram
\begin{equation}\label{g-ext}
\xymatrix
{
\tilde{X} \ar[dd]_{\tilde{\pi}}\ar[dr]^{\phi}  & \\
 & X\ar[dl]^{\pi}\\
Y &
}
\end{equation}
where $\tilde\pi: \tilde X \to Y\cong \tilde X/ K$ is
a group extension and so is the extension
$\phi: \tilde X \to X\cong \tilde X/L$
with $L=A/A_0 \subset F/A_0 = K$. ($\tilde X =X$ iff
$A$ is a normal subgroup of $F$.)

\br

A minimal system \XGa is called {\em incontractible}
if the union of minimal subsets is dense in every
product system $(X^n, \Ga)$. This is the case iff
$p\circ Gx=X$ for some (hence every) $x\in X$
and $p\in M$. When $\Ga$ is abelian $Gx$ is always dense
in $X$ so that every minimal system is incontractible.
However the following relative notion is an important tool
even when $\Ga$ is abelian.

We say that \XtoYGa is a RIC ({\em relatively incontractible})
{\em extension} if:
\begin{quote}
for every $y=py_0\in Y$, $p$ an element of $M$,
$\pi^{-1}(y)=p\circ u\pi^{-1}(y_0)=
p\circ Fx_0$, where $F=\mathfrak{G}(Y,y_0)$.
\end{quote}
One can show that every RIC extension is open and that
every distal extension is RIC. It then
follows that every distal extension is open.

We have the following theorem from \cite{EGS}
about the interpolation of equicontinuous extensions.
For a proof see \cite{Gl1}, Theorem X.2.1.

\begin{thm}\label{eq-ext}
Let $\pi: X \to Y$ be a RIC extension of minimal systems.
Fix a point $x_0\in X$ with $ux_0=x_0$ and let $y_0=\pi(x_0)$.
Let $A=\mathfrak{G}(X,x_0)$ and $F=\mathfrak{G}(Y, y_0)$.
Then there exists a commutative diagram of pointed systems
\begin{equation}\label{eq-ext-d}
\xymatrix
{
(X,x_0) \ar[dd]_{\pi}\ar[dr]^{\sig}  & \\
 & (Z,z_0)\ar[dl]^{\rho}\\
(Y,y_0) &
}
\end{equation}
such that  $\rho$ is the largest intermediate 
equicontinuous extension. Its
Ellis group satisfies $\mathfrak{G}(Z,z_0)=AF'$. 
The extension $\rho$ is an isomorphism iff $AF'=F$.
\end{thm}

\br

Given a homomorphism $\pi:(X,\Ga) \to (Y,\Ga)$ of minimal metric systems,
there are several standard constructions of associated ``shadow diagrams".
In the {\em O-shadow diagram}
\begin{equation}\label{O}
\xymatrix
{
X \ar[d]_{\pi}  & X^* =X \vee Y^*\ar[l]_-{\tet^*} \ar[d]^{\pi^*} \\
Y  & Y^*\ar[l]^{\tet}
}
\end{equation}
the map $\pi^*$ is open
and the maps $\theta$ and $\theta^*$ are almost 1-1.
The explicit constructions is as follows.
The set valued map $\pi^{-1}: Y \to 2^X$ (where the latter is
the compact space of closed subsets of $X$, equipped with the
Hausdorff, or Vietoris, topology) is upper-semicontinuous
and we let $Y_0\subset Y$ be the set of continuity points
of this map. Set
$Y^*={\cls}\{\pi^{-1}(y): y\in Y_0\}\subset 2^X$,
and $X^*=X\vee Y^* = {\cls}\{(x,\pi^{-1}(y)): y\in Y_0,\ \pi(x)=y\}
\subset X\times Y^*$.
By the upper-semicontinuity of $\pi^{-1}$ every $y^*\in Y^*$
is contained in a fiber $\pi^{-1}(y)$ for some $y\in Y$
and we let $\tet(y^*)=y$.
The maps $\pi^*$ and $\tet^*$ are the restriction to $X^*$
of the coordinate projections on $X$ and $Y^*$ respectively.
One then shows that
$X^*=\{(x,y^*): x \in y^*\in Y^*\}$
and that indeed, $\pi^*$ is open
and the maps $\theta$ and $\theta^*$ are highly proximal.
The O-shadow diagram collapses, i.e. $Y=Y^*$, $X=X^*$
and $\pi=\pi^*$ iff $\pi: X\to Y$ is an open map; iff
the map $\pi^{-1}: Y \to 2^X$ is continuous.

In the {\em RIC-shadow diagram}
\begin{equation}
\xymatrix
{
X \ar[d]_{\pi}  & X^*=X\vee Y^* \ar[l]_-{\tet^*} \ar[d]^{\pi^*} \\
Y  & Y^*\ar[l]^{\tet}
}
\end{equation}
$ \pi^*$ is RIC and $\theta, \theta^*$
are proximal (thus we still have
$A={\mathfrak{G}}(X,x_0)={\mathfrak{G}}( X^*, x^*_0)$ and
$F={\mathfrak{G}}(Y,y_0)= {\mathfrak{G}}(Y^*, y^*_0)$).
The concrete description of these objects uses
quasifactors and the circle operation:
$$
Y^*= \{p\circ Fx_0:p\in M\}\subset  2^X,\qquad
X^*= \{(x, y^*):x\in  y^*\in Y^*\}
\subset X\times  Y^*
$$
and
$$
\theta(p\circ Fx_0)=py_0, \quad
\theta^*(x, y^*)=x,\quad
\pi^*(x, y^*)= y^*,\quad (p\in M),
$$
where $F={\mathfrak{G}}(Y,y_0)$.
The map $\theta$ is an isomorphisms (hence $\pi=\pi^*$)
when and only when $\pi$ is already RIC.

\br

We say that $\pi:(X,\Ga) \to (Y,\Ga)$ has a
{\em relatively invariant measure} (or that is a RIM extension),
if there exists a projection $P:C(X) \to C(Y)$ such that
\begin{enumerate}
\item
$P(f) \ge 0$ for $f\ge 0$ in $C(X)$.
\item
$P(\ch)=1.$
\item
$P(h \circ \pi)= h$ for every $h \in C(Y)$.
\item
$P(f\circ \ga) = P(f) \circ \ga$ for every $f\in C(X)$ and $\ga\in \Ga$.
\end{enumerate}
This property is equivalent to the existence of a continuous {\em section},
i.e. a continuous $\Ga$ equivariant map $y \mapsto \la_y$
from $Y$ into $\Mf(X)$ such that $\pi(\la_y)=\del_y$ for every $y\in Y$.
Here and in the sequel we use the same letter $\pi$ to denote the induced
map $\pi: \Mf(X) \to \Mf(Y)$ on the spaces of probability measures.
Sometimes though we will write $\pi_*$ for the induced map.

In the sequel we will need the following lemmas.

\begin{lem}\label{RIM-fs}
Let $\pi:(X,\Ga) \to (Y,\Ga)$ be a RIM homomorphism of minimal metric systems
with section $y \mapsto \la_y$. Then there exists a dense $G_\del$ subset
$Y_{fs} \subset Y$ with the property that ${\supp}(\la_y) = \pi^{-1}(y)$
for every $y \in Y_{fs}$.
\end{lem}

\begin{proof}
It is easy to see that, in general, the map $y \mapsto {\supp}(\la_y)$ from
$Y$ into $2^X$ is lower-semicontinuous; i.e. whenever $y_i \to y$ in $Y$
then $\liminf {\supp}(\la_{y_i}) \supset {\supp}(\la_y)$. Such a map admits
a dense $G_\del$ set of continuity points which we denote by $Y_{fs}$.
We claim that for each $y \in Y_{fs}$ the measure $\la_y$ has {\em full support}, i.e.  
${\supp}(\la_y) = \pi^{-1}(y)$.
To see this observe first that, by minimality, ${\cls} (\cup\{{\supp}(\la_y) : y \in Y\})=X$.
Now if $y \in Y_{fs}$ and $x \in  \pi^{-1}(y)$ then there is a sequence $x_i \to x$
with $x_i \in {\supp}(\la_{y_i})$, for some sequence $y_i \in Y$. Necessarily
$\pi(x_i) = y_i \to y$ and, as $y \in Y_{fs}$, we have $\lim {\supp}(\la_{y_i}) = {\supp}(\la_y)$.
Now
$$
x \in \liminf {\supp}(\la_{y_i}) = \lim {\supp}(\la_{y_i}) = {\supp}(\la_y).
$$
\end{proof}

\begin{lem}\label{ORIM}
Let $\pi:(X,\Ga) \to (Y,\Ga)$ be a RIM homomorphism of minimal metric systems
with section $y \mapsto \la_y$. Let (\ref{O}) be the associated O-shadow diagram.
Then the open homomorphism $\pi^*$ is RIM as well. 
\end{lem}

\begin{proof}
We use the notations following the definition of the O-shadow diagram (\ref{O}).
Let $Y_0 \ni y_i \to y \in Y \setminus Y_0$ and let $Y^* \ni y^* = \lim \pi^{-1}(y_i)$.
By the lower-semicontinuity of the map $\la_y \mapsto {\supp}(\la_y)$
we see that ${\supp}(\la_y) \subset y^*$. Thus for every
such $y$ ${\supp}(\la_y) \subset \bigcap \{y^* \in Y^*:  y^* \subset \pi^{-1}(y)\}$.
It is now easy to check that the map
$y^* \mapsto \la_{\tet(y^*)} \times \del_{y^*}$ is a section for $\pi^*$.
\end{proof}

Given a homomorphism $\pi:(X,\Ga) \to (Y,\Ga)$ of minimal metric systems there is an associated  {\em RIM shadow diagram}
\begin{equation}\label{RIM}
\xymatrix
{
X \ar[d]_{\pi}  & \tilde{X}=X \vee \tilde{Y}
 \ar[l]_-{\tilde{\tet}} \ar[d]^{\tilde{\pi}} \\
Y  & \tilde{Y}\ar[l]^{\tet}
}
\end{equation}
the map $\tilde \pi$ has a RIM
and the maps $\theta$ and $\tilde\theta$ are strongly proximal.
It can be shown that every isometric extension has a RIM and is open.
See \cite{G0} for more details, also a treatment of SPI systems can
be found in \cite{G1}.

\br

We say that a minimal system \XGa is a
{\em strictly PI system} if there is an ordinal $\eta$
(which is countable when $X$ is metrizable)
and a family of systems
$\{(W_\iota,w_\iota)\}_{\iota\le\eta}$
such that (i) $W_0$ is the trivial system,
(ii) for every $\iota<\eta$ there exists a homomorphism
$\phi_\iota:W_{\iota+1}\to W_\iota$ which is
either proximal or equicontinuous
(isometric when $X$ is metrizable), (iii) for a
limit ordinal $\nu\le\eta$ the system $W_\nu$
is the inverse limit of the systems
$\{W_\iota\}_{\iota<\nu}$,  and
(iv) $W_\eta=X$.
We say that \XGa is a {\em PI-system} if there
exists a strictly PI system $\tilde X$ and a
proximal homomorphism $\theta:\tilde X\to X$.

If in the definition of (strictly) PI-systems we replace
proximal extensions by HP extensions (almost 1-1 extensions
in the metric case) we get the notion of (strictly) {\em HPI} ({\em AI-systems}
in the metric case).
If we replace the proximal extensions by trivial
extensions (i.e.\ we do not allow proximal
extensions at all) we have {\em I-systems}.
In this terminology the structure theorem for distal
systems (Furstenberg \cite{F}, 1963) can be stated as follows:
\begin{thm}\label{FST}
A metric minimal system is distal iff it is an I-system.
\end{thm}

And, the Veech-Ellis structure theorem for point distal systems
(Veech \cite{V}, 1970 and Ellis \cite{E1}, 1973):

\begin{thm}
A metric minimal dynamical system is point distal iff it is an
AI-system.
\end{thm}

The structure theorem for the general minimal system is proved in
\cite{EGS} and \cite{McM} (see also \cite{V}) and asserts that
every minimal system admits a canonically defined proximal
extension which is a weakly mixing RIC extension of a strictly PI system.
Both the Furstenberg and the Veech-Ellis structure theorems are
corollaries of this general structure theorem.

\begin{thm}[Structure theorem for minimal systems]\label{structure}
Given a minimal system $(X,\Ga)$, there exists an ordinal $\eta$
(countable when $X$ is metrizable) and a canonically defined
commutative diagram (the canonical PI-Tower)
\begin{equation*}
\xymatrix
        {X \ar[d]_{\pi}             &
     X_0 \ar[l]_{{\theta}^*_0}
         \ar[d]_{\pi_0}
         \ar[dr]^{\sigma_1}         & &
     X_1 \ar[ll]_{{\theta}^*_1}
         \ar[d]_{\pi_1}
         \ar@{}[r]|{\cdots}         &
     X_{\nu}
         \ar[d]_{\pi_{\nu}}
         \ar[dr]^{\sigma_{\nu+1}}       & &
     X_{\nu+1}
         \ar[d]_{\pi_{\nu+1}}
         \ar[ll]_{{\theta}^*_{\nu+1}}
         \ar@{}[r]|{\cdots}         &
     X_{\eta}=X_{\infty}
         \ar[d]_{\pi_{\infty}}          \\
        pt                  &
     Y_0 \ar[l]^{\theta_0}          &
     Z_1 \ar[l]^{\rho_1}            &
     Y_1 \ar[l]^{\theta_1}
         \ar@{}[r]|{\cdots}         &
     Y_{\nu}                &
     Z_{\nu+1}
         \ar[l]^{\rho_{\nu+1}}          &
     Y_{\nu+1}
         \ar[l]^{\theta_{\nu+1}}
         \ar@{}[r]|{\cdots}         &
     Y_{\eta}=Y_{\infty}
    }
\end{equation*}
where for each $\nu\le\eta, \pi_{\nu}$
is RIC, $\rho_{\nu}$ is isometric, $\theta_{\nu},
{\theta}^*_{\nu}$ are proximal and
$\pi_{\infty}$ is RIC and weakly mixing.
For a limit ordinal
$\nu ,\  X_{\nu}, Y_{\nu}, \pi_{\nu}$
etc. are the inverse limits (or joins) of
$ X_{\iota}, Y_{\iota}, \pi_{\iota}$ etc. for $\iota
< \nu$.
Thus $X_\infty$ is a proximal extension of $X$ and a RIC
weakly mixing extension of the strictly PI-system $Y_\infty$.
The homomorphism $\pi_\infty$ is an isomorphism (so that
$X_\infty=Y_\infty$) iff $X$ is a PI-system.
\end{thm}

Two useful criteria on minimal systems for being PI and HPI 
are given in theorems of Bronstein \cite{Bro}
and  van der Woude \cite{vdW}, respectively.

In the sequel we will mostly deal with metrizable systems. So from now on,
unless we explicitly say
otherwise {\bf all the dynamical systems will be assumed to be metrizable.}

\br

All of the above mentioned theorems have their relative versions.
One starts with a homomorphism $\pi : (X, \Ga) \to (Y, \Ga)$ between
minimal systems and then all the various notions defined above,
equicontinuity, distality, point-distality etc. have natural relative analogues.

We will use two such relative theorems which I now proceed to describe.

\begin{defn}\label{Y-pd-def}
Let $\pi : X \to Y$ be a homomorphism of minimal systems. 
Then $\pi$ is a {\em point distal extension} if
there is a point $x_0 \in X$ such that
$P[x_0] \cap \pi^{-1}(y_0) =\{x_0\}$, where $y_0 = \pi(x_0)$.
A point such as $x_0$ is called a {\em $Y$-distal point}.
\end{defn}
%

%

\begin{thm}[Veech \cite{V0}, Ellis \cite{E3}]\label{PD-rel}
Let $\pi : X \to Y$ be a homomorphism of metric minimal systems. 
Then $\pi$ is point distal iff $\pi$ is an AI extension and the
set of $Y$-distal points is a dense $G_\del$ subset of $X$.
If $\pi$ is open then it is point distal iff it is strictly AI.
\end{thm}
%

A homomorphism $\pi : (X,\Ga) \to (Y,\Ga)$ is called {\em semiopen}
if the interior of $\pi(U)$ is nonempty for every nonempty open subset
$U$ of $X$. When $X$ is minimal every
$\pi : (X,\Ga) \to (Y,\Ga)$ is semiopen (see Lemma \ref{so-min} below).
We will say that a subset $W \subset X\times X$ is a {\em S-set}
if it is closed, invariant, topologically transitive, and the restriction
to $W$ of the projection maps are semiopen.



\begin{thm}[van der Woude,  \cite{vdW}]\label{vdw-rel}
A homomorphism $\pi : X \to Y$ between metric minimal systems $(X,\Ga)$ and $(Y,\Ga)$ 
is a point distal extension iff every S-set of $R_\pi  = \{(x, x') \in X \times X : \pi(x) = \pi(x')\}$ 
is minimal.
\end{thm}

\br

\section{On semiopen maps}\label{so-sec}

\begin{lem}\label{so-lem}
Let $\pi: X\to Y$ be a continuous surjection between
compact Hausdorff spaces. The conditions {\rm 1} and {\rm 2}
below are equivalent.
If $X$ is metrizable then the three conditions
are equivalent:
\begin{enumerate}
\item
$\pi$ is semiopen.
\item
The preimage of every dense subset in $Y$ is dense in $X$.
\item
The set
\begin{equation*}
X_0 = \{x\in X: \text{ the set valued map $\pi^{-1}: Y \to 2^X$
is continuous at $\pi(x)$}\}
\end{equation*}
is dense in $X$. 
\end{enumerate}
\end{lem}

\begin{proof}
The equivalence of 1 and 2 is straightforward. For any
continuous surjection $\pi: X\to Y$ the corresponding set map
$\pi^{-1}: Y \to 2^X$ is upper-semicontinuous and, when $X$ is metrizable,
this implies that it has a dense $G_\del$ subset $Y_0\subset Y$ of
continuity points.
Assuming 2 we conclude that $X_0=\pi^{-1}(Y_0)$ is a dense $G_\del$
subset of $X$. Conversely if 3 is valid and $U\subset X$ is
open and nonempty, then $U \cap X_0 \not=\emptyset$ and if $x_0$
is any point in this intersection then $\pi$ is open at $x_0$,
so that $\pi(U)$ is a neighborhood of $\pi(x_0)$ and we conclude
that $\pi(U)^\circ \not=\emptyset$.
\end{proof}

\begin{lem}\label{so-res-lem}
Let $X \overset{\al}{\to} Z \overset{\beta}{\to} Y$, where $X$, $Y$ and
$Z$ are compact metric spaces, $\al$ and $\beta$ are continuous surjections,
and $\al$ is semiopen. For $y \in Y$ set $Z_y = \beta^{-1}(y)$ and
$X_y = (\beta \circ \al)^{-1}(y) = \al^{-1}(Z_y)$.
Then there is a dense $G_\del$ subset $Y_0 \subset Y$
such that for every $y \in Y_0$ the restriction $\al : X_y \to Z_y$ is semiopen.
\end{lem}

\begin{proof}
Let $V \subset X$ be a nonempty open subset.
Set 
$$
Y_V = \{y \in Y : \al(V)^\circ \cap Z_y \not = \emptyset\}\
\cup \ \{y \in Y : \ol{V} \cap X_y = \emptyset\} : =
Y_{V,1} \cup Y_{V,2}.
$$
We will show that $V_Y$ is open and dense.

Suppose $y_n \to y$ is a convergent sequence with $y_n \not\in Y_V$.
Then, as  $Y_V^c = Y_{V,1}^c \cap Y_{V,2}^c$, we have
\begin{itemize}
\item
For all $n$
\begin{gather*}
y_n \in Y_{V,1}^c \ \imp \ \al(V)^\circ \cap Z_{y_n} = \emptyset
 \ \imp \ Z_{y_n} \subset (\al(V)^\circ)^c \ \imp \ \\
Z_y \subset (\al(V)^\circ)^c \ \imp\ 
\al(V)^\circ \cap Z_y = \emptyset
 \ \imp\ 
y \in Y_{V,1}^c.
\end{gather*}

\item
And, for all $n$
\begin{gather*}
y_n \in Y_{V,2}^c \ \imp \ \ol{V} \cap X_{y_n} \not = \emptyset 
\ \imp \  \ol{V}  \cap \limsup X_{y_n} \not= \emptyset \ \imp \ \\
\ol{V} \cap X_y  \not= \emptyset  \ \imp \ 
y \in Y_{V,2}^c.
\end{gather*}
\end{itemize}
This shows that $Y_V$ is open.

Next let $O \subset Y$ be a nonempty open set. We will show that
$O \cap Y_V \not = \emptyset$.
\begin{itemize}
\item
Case1 : Suppose $\beta^{-1}(O) \cap \ol{\al(V)} \not=\emptyset$.

Then $\beta^{-1}(O) \cap (\al(V)^\circ \cup \partial(\al(V)) \not=\emptyset$.
As $\al$ is a semiopen map, the set $\partial(\al(V))$ has an empty interior
and it follows that $\beta^{-1}(O) \cap \al(V)^\circ  \not=\emptyset$.
Let $z$ be a point in this intersection and set $y = \beta(z)$. Then
$y \in O$ and $z \in  \al(V)^\circ \cap Z_y$ imply that $O \cap Y_{V,1} \not=\emptyset$.
\item
Case 2 : Suppose $\beta^{-1}(O) \cap \ol{\al(V)} =\emptyset$.

It then follows that for every $y \in O$ we have $\ol{V} \cap X_y = \emptyset$,
whence $y \in O \cap Y_{V,2}$.
\end{itemize}

Let now $\{V_i\}_{ \in \N}$ be an enumeration of a basis for open sets on $X$.
Set $Y_i = Y_{V_i}$ and let $Y_0 = \cap_{i \in \N} Y_i$. 

We now check that
the set $Y_0 \subset Y$ satisfies the requirement of the lemma.
In fact, suppose $y \in Y_0$. We want to show that $\al : X_y \to Z_y$ is semiopen.
Let then $U \subset X_y$ be a nonempty open subset. With no loss
of generality we assume that $U = V \cap X_y$ for some $V$ in our basis $\{V_i\}_{i \in \N}$.
In particular then $V \cap X_y \not=\emptyset$, whence
$y \not \in Y_{V,2}$. But $y \in Y_0 \subset Y_V$, hence $y \in Y_{V,1}$,
i.e. $\al(V)^\circ \cap Z_y \not=\emptyset$.
If $z$ is a point in this intersection then $z = \al(x)$ for some $x \in V \cap X_y = U$.
Thus $\al(U)^\circ \not=\emptyset$ and our proof is complete.
\end{proof}

\begin{cor}\label{so-res}
Let $\pi: X \to Y$ be a continuous semiopen surjection between
compact metric spaces. Let $R_\pi = \{(x,x') \in X \times X : \pi(x) = \pi(x')\}$.
Suppose $W \subset R_\pi$ is a closed subset such that $(\pi \times \pi)(W) = \Del_Y
= \{(y, y) : y \in Y\}$.
Suppose further that the restrictions of the projection maps ${\bf{p}}_i : X \times X \to X \ (i=1,2)$
to $W$ are semiopen. Then there is a dense $G_\del$ subset $Y_{so} \subset Y$ such that
for each $y \in Y_{so}$ the maps ${\bf{p}}_i : W \cap (\pi^{-1}(y) \times \pi^{-1}(y)) \to \pi^{-1}(y)$
are semiopen.
\end{cor}

\begin{proof}
Apply Lemma \ref{so-res-lem} to $W \overset{{\bf{p}}_i}{\to} X \overset{\pi}{\to} Y$
with $i=1, 2$ and then take let $Y_{so}$ be the intersection of the two corresponding dense $G_\del$
sets of $Y$. 
\end{proof}

%
%
%
\br

A result of Ditor and Eifler from 1972, \cite{DE} asserts that a
continuous surjection $\pi: X\to Y$ between compact Hausdorff
spaces $X$ and $Y$ is open iff the induced map $\pi_*:
\Mf(X) \to \Mf(Y)$ is an open surjection.
In the course of the proof of our main theorem (Theorem \ref{main})
we will need the following analogous result (in the metric case) for semiopen maps.
For the proof we refer to \cite[Theorem 2.3]{Gmt}.

\begin{thm}\label{so-thm}
Let $\pi: X \to Y$ be a continuous surjection between compact
metric spaces which is semiopen. Then the induced map $\pi_*:
\Mf(X) \to \Mf(Y)$ is a semiopen surjection.
\end{thm}

Recall the following well known result;
for completeness we include a proof.

\begin{lem}\label{so-min}
Let $\pi: (X,\Ga) \to (Y,\Ga)$ be a homomorphism between minimal
systems. Then  $\pi$ is semiopen.
\end{lem}

\begin{proof}
Let $W\subset X$ be a closed set with nonempty interior.
By minimality of \XGa there is a finite set $\{\ga_1,\dots,\ga_n\}
\subset \Ga$ with $X = \bigcup_{i=1}^n \ga_i W$.
Therefore $Y = \bigcup_{i=1}^n \pi(\ga_i W)$
and it follows that for some $i$ the interior of the closed
set $\pi(\ga_i W) = \ga_i\pi(W)$ is nonempty. Thus, as required, also
$\pi(W)^\circ \not=\emptyset$.
\end{proof}

\br

\section{A key proposition on diffused measures}\label{diff-sec}
As explained above we fix a minimal ideal $M$ in $\beta \Ga$ and let
$u$ be an idempotent in $M$.
We denote the subgroup $uM$ of $M$ by $G$ and identify it
with the group of automorphisms of the universal $\Ga$-minimal
system $(M,\Ga)$, where for $\al\in G$ the corresponding
automorphism $R_\al: M \to M$ is given by right multiplication
$p \mapsto p \al$.
For an abelian $\Ga$ each subgroup $vM\subset M$, where $v$ is an idempotent
in $M$, is dense in $M$ and it follows that the $G$-dynamical
system $(M,G)$ (where $G$ acts by right multiplication) is minimal.
In the general case the compact dynamical system $(M, G)$ admits
a minimal subset and it follows that there is a minimal idempotent $v$
such that ${\cls}(vG)$ is minimal under the right $G$ action.
In the sequel we will need the following slightly stronger statement.

\begin{lem}\label{G-min-id}
If $(Y, \Ga)$ is a minimal proximal system then
\begin{enumerate}
\item
For every minimal idempotent $v \in M$
there is a unique point $y_0 \in Y$ with $vy_0 = y_0$ and 
moreover, $vy=y_0$ for every $y \in Y$.
\item
For every $y \in Y$ there is a minimal idempotent $v \in M$
such that $vy = y$ and the $G$-system $({\cls}(vG), G)$ is minimal.
\end{enumerate}
\end{lem}

\begin{proof}
The first assertion is clear. For the second observe that
the closed set $\{p \in M : py =y\}$ is $G$-invariant (under the 
right action of $G$) and hence contains a $G$-minimal set.
Clearly this set contains an idempotent, and any such 
idempotent  $v$ satisfies the required property.
\end{proof}

As our choice of $u$ was arbitrary we can, and will from now on, assume that
${\cls}(uG) = {\cls}(Gu) = {\cls}(G)$ is $G$-minimal.

\br

\begin{lem}\label{sat}
Let $(Z^*, \Ga)$ be a minimal metric system and let
$\phi : (Z^*, \Ga) \to (Y, \Ga)$ be its maximal proximal factor.
Suppose further that we have the following diagram
$$
Z^* \overset{\tet}{\to} Z \overset{\sig}{\to} Y,
$$
where $\sig$ is a maximal isometric (i.e. equicontinuous) extension,
$\tet$ is an almost 1-1 extension, 
and $\phi = \tet \circ \sig$ (i.e. $\phi$ is an almost automorphic extension).
Let $O$ be a nonempty open subset of $Z^*$.
\begin{enumerate}
\item
There is a nonempty open subset $V\subset O$ such that
$\tet^{-1}(\tet(V))\subset O$.
\item
There is a nonempty open subset $W\subset O$ such that
${\cls}_\tau(W \cap uZ^*) \subset O$.
\end{enumerate}
\end{lem}

\begin{proof}

1.\
Suppose $O \subset Z^*$ is a nonempty open set for which the statement of
the lemma fails. Choose a point $z^* \in O$ such that
$\tet^{-1}(\tet(z^*))=\{z^*\}$ and let $V_n\subset O$ be a sequence of
open balls centered at $z^*$ with ${\diam}(V_n) \searrow 0$.
By assumption there are pairs of points $z_n \in V_n$
and $z'_n\not\in O$ with $\tet(z_n)=\tet(z'_n)$. However,
as $\lim_{n\to\infty} \tet^{-1}(\tet(z_n)) = \{z^*\}$,
we have $z'_n \to z^*$ in contradiction of the fact that $O$
is a neighborhood of $z^*$.

2.\
Since $Z$ is an isometric extension of the proximal system $Y$, $uZ$ is a closed subset of $Z$
and the $\tau$-topology on $uZ$ coincides with its
compact Hausdorff group topology.
Since $\tet$ is an almost one-to-one map,
the restriction $\tet\rest uZ^* : uZ^* \to uZ$ is a
homeomorphism of $uZ^*$, equipped with the $\tau$ topology, onto $uZ$.
Let $O \subset Z^*$ be a nonempty open set. Let $V \subset O$
be as in part 1, and let $W$ be a
nonempty open subset such that $\ol{W} \subset V$.
Now
\begin{align*}
{\cls}_\tau (W \cap uZ^*) 
& = \tet^{-1}({\cls}_\tau(\tet(W \cap uZ)) \cap uZ^*\\
& = \tet^{-1}(\ol{\tet(W \cap uZ)}) \cap uZ^*\\
& \subset  \tet^{-1}(\ol{\tet(W)})
 = \tet^{-1}(\tet(\ol{W}))\\
& \subset  \tet^{-1}(\tet(V)) \subset O.
\end{align*}
\end{proof}

\begin{lem}\label{tau}
Let $(X, \Ga)$ be a minimal metric system and let
$\phi : (X,\Ga) \to (Y,\Ga)$ be its maximal proximal factor.
Suppose further that we have the following diagram
$$
X \overset{\rho}{\to} Z^* \overset{\tet}{\to} Z \overset{\sig}{\to} Y,
$$
where $\rho$ and $\sig$ are maximal isometric extensions,
$\tet$ is an almost 1-1 extension, 
and $\phi = \sig \circ \tet \circ \rho$.
%
%
%
%
Fix a minimal ideal $M \subset \beta\Ga$ and an idempotent $u\in M$ as above.
Let $U$ be an open subset of $X$ such that $U \cap uX \not= \emptyset$.
Then
$$
{\cls}_\tau (U\cap uX) \supset \rho^{-1}(\rho(U)) \cap uX.
$$
\end{lem}

%
%
\begin{proof}
Fix a point $x_0 \in X$ with $ux_0=x_0$ and let $z^*_0$,  $z_0$ and $y_0$
be its images in $Z^*$, $Z$ and $Y$ respectively.
Let
\begin{gather*}
A=\mathfrak{G}(X,x_0)=\{\alpha\in G:\alpha x_0=x_0\}, \ \text{and}\ \\
F = \mathfrak{G}(Z,z_0)=
\mathfrak{G}(Z^*,z^*_0) = \{\alpha\in G:\alpha z_0=z_0\} =
\{\alpha\in G:\alpha z^*_0=z^*_0\}.
\end{gather*}
Note that, as $Y$ is a proximal system, we have $\mathfrak{G}(Y, y_0) = G$.
The assumption that $\rho$ is an isometric extension implies that
$B = F/A$ is a homogeneous space of the Hausdorff compact topological group
$F/\cap\{f A f^{-1} : f \in F\}$ (with respect to its $\tau$-topology).
The fact that $(Z,\Ga)$ is the maximal equicontinuous extension of $(Y, \Ga)$
(within $X$) implies that $F \supset G'$ and that $G'A=AG'=F$.
Let $h: M \to X$ denote the evaluation map $p \mapsto px_0$.

Let $U$ be a nonempty open subset of $X$
such that $U \cap uX \not= \emptyset$.
Set $\tilde{U}=h^{-1}(U)=
\{p\in M: px_0 \in U\}$.
Then $\tilde{U}$ is an open subset of $M$ with
 $\tilde{U} \cap G \not=\emptyset$
and,
by minimality of the $G$-system $({\cls} {(G)},G)$,
the collection $\{\tilde{U}\al: \al \in G\}$ is an open cover of ${\cls}{G}$.
Choose a finite subcover, say $\{\tilde{U}\al_i: i=1,2,\dots,n\}$.
Now
$$
\bigcup_{i=1}^n  {\cls}_\tau(\tilde{U}\al_i \cap G)
= \bigcup_{i=1}^n {\cls}_\tau(\tilde{U} \cap G)\al_i = G;
$$
hence ${\cls}_\tau(\tilde{U}\cap G)$ has a nonempty $\tau$-interior.
Since ${\cls}_\tau(\tilde{U} \cap G)$ is also $\tau$-closed, it must
contain a left translate of $G'$, say $\beta G'$ for some $\beta \in G$
(this follows from the definition of $G'$, see equation (\ref{prime})).
Projecting back to $X$ via $h$ we get
\begin{align*}
{\cls}_\tau (U \cap uX ) & =
h({\cls}_\tau (\tilde{U}\cap G)) \\
&\supset
\beta G' x_0 = \beta G' A x_0 =
\beta Fx_0 \\
& = (\beta F \beta^{-1})\beta x_0 = \rho^{-1}(\rho(\beta x_0)).
\end{align*}

Let $O=\rho(U)$, then $O$ is a nonempty open subset of $Z^*$
and by Lemma \ref{sat} there is a nonempty open subset $W\subset O$
such that ${\cls}_\tau(W \cap uZ^*) \subset O$. Set $U_1=\rho^{-1}(W)
\cap U$. Then $U_1$ is a nonempty open subset of $X$ and by the above argument
there exists $\beta_1\in G$ with $\rho^{-1}(\rho(\beta_1 x_0))
\subset {\cls}_\tau (U_1 \cap uX )\subset {\cls}_\tau (U \cap uX )$. Now
\begin{align*}
\rho(\beta_1 x_0) &
\in \rho({\cls}_\tau (U_1 \cap uX ))\\
& = {\cls}_\tau (\rho(U_1) \cap uZ^*)\\
& \subset {\cls}_\tau (W \cap uZ^* ) \subset O = \rho(U).
\end{align*}

Thus we have shown that for every nonempty open subset $U\subset X$
which meets $uX$,
the set ${\cls}_\tau (U \cap uX )$ contains a full fiber
$\rho^{-1}(z^*)$ for some $z^*\in \rho(U)$. Since $\rho$ is an open map
we conclude that
$$
{\cls}_\tau (U\cap uX) \supset \rho^{-1}(\rho(U)) \cap uX
$$
as required.
\end{proof}

\br

This lemma has the following surprizing corollary.

\begin{cor}\label{finite}
In the situation described in Lemma \ref{tau}, if $\sig : Z \to Y$ is a finite to one isometric extension
then $\rho : X \to Z^*$ is an isomorphism, i.e. 
$\phi : X \to Y$ (with $\phi = \sig \circ \tet \circ \rho$)
 is an almost automorphic extension. 
In particular if, in addition,
$Y$ is the trivial one point system, then $X = Z$ is a finite system.
\end{cor}

\begin{proof}
Let $z_0 \in Z$ be a point with $\tet^{-1}(z_0) = \{z_0^*\}$ a singleton.
With no loss of generality we can assume that $uz_0 =z_0$
(see Lemma \ref{G-min-id}).
Note that then also $uz_0^* = z_0^*$ and $\rho^{-1}(z_0^*)
\subset uX$.
Let $y_0 =uy_0 = \sig(z_0)  \in Y$ and let $X_0 = \phi^{-1}(y_0)$, $Z^*_0= \rho(X_0)$ and 
$Z_0 = \sig^{-1}(y_0) = uZ$.
By our assumptions $Z_0$ is a finite set and therefore
$\{z^*_0\}$  is a relatively open subset of u$Z^*_0$. 
Then $\rho^{-1}(z^*_0)$ is a relatively open subset of $X_0$. 
We claim that $\rho^{-1}(z^*_0)$ is a singleton.
If this is not the case,
then there is an open subset $U \subset X$ with $\emptyset
\not = \ol{U} \cap \rho^{-1}(z^*_0)  \subsetneq \rho^{-1}(z^*_0)$.
By Lemma \ref{tau} we have
$$
{\cls}_\tau (U \cap uX) \supset \rho^{-1}(\rho(U)) \cap uX.
$$
But ${\cls}_\tau (U \cap uX)   \subset \ol{U}$ (because $\rho$ 
is an isometric extension) and
$z^*_0 \in \rho(U) $, whence 
$\rho^{-1}(\rho(U)) \cap uX \supset  \rho^{-1}(z^*_0)$. It follows
that the set $\ol{U}$ contains the whole fiber
$\rho^{-1}(z^*_0)$, a contradiction.
Thus $\rho^{-1}(z^*_0)$ is indeed a singleton and it follows that $\rho$
is one-to-one as claimed.
\end{proof}

\br

\begin{prop}\label{diff}
Let \XGa be a minimal metric system and let
$\phi : (X,\Ga) \to (Y,\Ga)$ be its maximal proximal factor.
Suppose further that we have the following diagram
$$
X \overset{\rho}{\to} Z^* \overset{\tet}{\to} Z \overset{\sig}{\to} Y,
$$
where $\rho$ and $\sig$ are maximal isometric extensions,
$\tet$ is an almost 1-1 extension, 
and $\phi = \sig \circ \tet \circ \rho$.
\begin{enumerate}
\item
For every $z^*\in Z^*$ the fiber $\rho^{-1}(z^*)$ has the structure of a
homogeneous space of a compact Hausdorff topological group
and we let $\la_{z^*}$ be the corresponding Haar measure on
this fiber.
In particular $\rho$ is a RIM and open extension and
$z^*\mapsto \la_{z^*}$, \ $Z^* \to \Mf(X)$, is the corresponding section.
Let $\La: \Mf(Z^*) \to \Mf(X)$, defined by
\begin{equation*}
\La(\nu) =\int_{Z^*} \la_{z^*}\, d\nu(z^*),
\end{equation*}
be the associated affine injection.
\item
Let $y_0 \in Y$ be the unique point with $uy_0=y_0$. Denote
$Z_0 = \sig^{-1}(y_0)$, 
$X_0 = \phi^{-1}(y_0)$ and $Z_0^* = (\sig \circ \tet)^{-1}(y_0) = \rho(X_0)$.
Set $\Mf_m(X_0)=\{\La(\nu): \nu \in \Mf(Z_0^*)\}$.
Then the set
$$
R=\{\nu \in \Mf(X_0): \text{the orbit closure of $\nu$ meets $\Mf_m(X_0)$}\}
$$
is a dense $G_\del$ subset of $\Mf(X_0)$.
\end{enumerate}
\end{prop}

\begin{proof}
Part 1 is well known; see e.g. Corollary 3.7 in \cite{G0}.

\br

2.\ 
Fix a compatible metric $d$ on $\Mf(X)$.
Let $\kappa \in \Mf(X_0)$ and $\ep, \eta >0$ be given.
Find an atomic measure $\la=\frac{1}{n}\sum_{i=1}^n \del_{x_i},\ x_i \in X_0$
such that $d(\ka,\la) < \ep/2$.
Choose open disjoint neighborhoods $U_i$ of $x_i$,
so small that every measure of the form
$\mu=\frac{1}{n}\sum_{i=1}^n \mu_i,\ \mu_i \in \Mf(X_0)$
 with ${\supp}\mu_i \subset U_i$,
will satisfy $d(\mu,\la) < \ep/2$, and hence also
$d(\mu,\ka) < \ep$.

Set $\nu=\frac{1}{n}\sum_{i=1}^n \la_{z^*_i}$ with $z^*_i=\rho(x_i)$
(an element of $\Mf_m(X_0)$).
For each $z^*_i$ choose points $\{x'_{i,j}\}_{j=1}^k
\in \rho^{-1}(z^*_i)$ so that $d(\mu',\nu) < \eta/2$,
where
$$
\mu'=\frac{1}{nk}\sum_{i=1}^n \sum_{j=1}^k \del_{x'_{i,j}}.
$$

By Lemma \ref{tau}
\begin{gather*}
uX \cap u \circ \left(uX \cap  \bigcup_{i=1}^n U_i\right)
= {\cls}_\tau \left(uX \cap  \bigcup_{i=1}^n U_i\right)\\
 \supset uX \cap \left(\bigcup_{i=1}^n \rho^{-1}(\rho(U_i))\right).
\end{gather*}

Therefore there exist an element $\ga \in \Ga$ and for each $i$ a
set $\{x_{i,j}\}_{j=1}^k \subset uX \cap U_i$, such that
$d(\ga x_{i,j},x'_{i,j})$ is so small that the inequality
$d(\ga \mu, \mu')< \eta/2$ is satisfied, with
$$
\mu=\frac{1}{nk}\sum_{i=1}^n \sum_{j=1}^k \del_{x_{i,j}}.
$$
Thus $d(\ga\mu,\nu)< \eta$. 
By the choice of the small sets $U_i$ we also have $d(\mu,\ka)< \ep$
and, as $\ep>0$ is arbitrary,
we have shown that the open set
$$
R_\eta=
\{\mu\in \Mf(X_0): \ {\text{there exists}}\ \ga\in \Ga\ {\text{ with}}\
d(\ga\mu, \Mf_m(X_0)) < \eta\}
$$
is dense in $\Mf(X_0)$.
Clearly $R=\bigcap\{R_\eta: \eta > 0\}$ is the required
dense $G_\del$ subset of $\Mf(X_0)$.
\end{proof}

\br

\section{Some properties of tame minimal systems}\label{tame-sec}

\begin{thm}\label{tame-inj}(\cite{G})
Let \XGa be a metric tame dynamical system.
Let $\Mf(X)$ denote the compact convex
set of probability measures on $X$ (with the weak$^*$
topology). Then each element $p\in E(X,\Ga)$ defines
an element $p_*\in E(\Mf(X),\Ga)$ and the map
$p \mapsto p_*$ is both a dynamical system and a semigroup
isomorphism of $E(X,\Ga)$ onto $E(\Mf(X),\Ga)$.
\end{thm}

\begin{proof}
Since $E(X,\Ga)$ is Fr\'echet we have for every $p\in E$ a sequence
$\ga_i\to p$ of elements of $\Ga$ converging to
$p$. Now for every $f\in C(X)$ and every
probability measure $\nu\in \Mf(X)$ we get, by the
Riesz representation theorem and
Lebesgue's dominated convergence theorem,
$$
\ga_i\nu(f)=\nu(f\circ \ga_i)\to \nu(f\circ p):=p_*\nu(f).
$$
Since the Baire class 1 function $f\circ p$ is well defined
and does not depend upon the choice of the convergent
sequence $\ga_i\to p$, this defines the map $p \mapsto p_*$ uniquely.
It is easy to see that this map is an isomorphism of dynamical systems,
whence also a semigroup isomorphism. Finally as $\Ga$ is dense in both
enveloping semigroups, it follows that this isomorphism is onto.
\end{proof}

As we have seen, when \XGa is a metrizable tame system,
the enveloping semigroup $E(X,\Ga)$ is a separable Fr\'echet
space. Therefore, each element $p\in E$
is a limit of a sequence of elements of $\Ga$,
$p=\lim_{n\to \infty} \ga_n$. It follows that the subset
$C(p)$ of continuity points of each $p\in E$ is a
dense $G_\del$ subset of $X$. More generally, if $A\subset X$
is any closed subset then the set $C_A(p)$ of continuity points
of the map $p\rest A : A \to X$ is a dense $G_\del$ subset
of $A$. For an idempotent $v=v^2\in E$ we write
$C_v$ for $C_{\ol{vX}}(v)$.
For a special case of the next lemma see \cite[Theorem 9.2]{GM14}

%
%
%
%

\br

\begin{lem}\label{Cp-rel}
Let $(X,\Ga)$ be a metrizable tame dynamical system
and $\pi : X \to Y$ a RIM extension, with section $y \mapsto \nu_y$,
such that $(Y, \Ga)$ is a proximal system.
\begin{enumerate}
\item
$p \nu_y =\nu_{py}$ for every $p \in E(X,\Ga)$ and every $y \in Y$.
In particular, we have $\nu_{py}(pX) =1$.
\item
For every minimal idempotent $v \in E(X,\Ga)$ there is a unique $y \in Y$ such that $vy=y$.
For such $v$ we have $vX \subset \pi^{-1}(y)$, so that
$\nu_y(vX)=1$ and 
${\supp}(\nu_y) \subset vX \cap \pi^{-1}(y) = v\pi^{-1}(y)$.
\item
For an idempotent $v$ in $E(X,\Ga)$ we have $C_v \subset vX$.
\item
If $X$ is minimal then for every point $y$ of the dense $G_\del$ subset $Y_{fs} \subset Y$
(see Lemma \ref{RIM-fs})
and $v$ a minimal idempotent in $E(X,\Ga)$ with $vy=y$,
$C_v$ is a dense $G_\del$ subset of $\pi^{-1}(y)$,
and $vX$ is a residual subset of $\pi^{-1}(y)$.
\end{enumerate}
\end{lem}

\begin{proof}
1.\
As $p : X \to X$ is a Baire 1 function, the measure $p\nu_x$ is
well defined and the equality follows 
from the continuity of the section.

2.\ Clear.

3.\
Given $x \in C_v$ choose a sequence $x_n \in vX$ with
$\lim_{n\to\infty} x_n =x$.
We then have $vx=\lim_{n\to\infty} vx_n = \lim_{n\to\infty} x_n =x$,
hence $C_v \subset vX$.

4.\
For $vy = y \in Y_{fs}$ we have ${\supp}(\nu_y) = \pi^{-1}(y)$.
By part 2 we have $\nu_y(vX) =1$, whence $vX$ is dense in $\pi^{-1}(y)$.
Now $C_v = C_{\ol{vX}}(v) = C_{\pi^{-1}(y)}(v)$ and it follows that $C_v$
is a dense $G_\del$ subset of $\pi^{-1}(y)$.
By part 3 $C_v \subset vX$ and it follows that $vX$ is residual in $\pi^{-1}(y)$.

\end{proof}

\begin{cor}
Let $(X,\Ga)$ be a metrizable tame dynamical system
admitting a $\Ga$ invariant probability measure, and
let $v$ be a minimal idempotent in $E(X,\Ga)$. Then
$C(v) \subset vX$, $C(v)$ is a dense $G_\del$ subset of $X$,
and $vX$ is residual in $X$.
\end{cor}

\br

\begin{thm}\label{PD}
Let $(X,\Ga)$ be a metrizable tame dynamical system
and $\pi : X \to Y$ a RIM extension, with section $y \mapsto \nu_y$,
such that $(Y,\Ga)$ is a proximal system.
Then the extension $\pi$ is point distal.
In particular, a metric tame minimal system admitting an invariant 
probability measure is point distal.
\end{thm}

\begin{proof}
We will prove that the condition in Theorem \ref{vdw-rel} holds; i.e.
that every $S$-set in $R_\pi$ is minimal. 
So let $W\subset R_\pi$ be an $S$-set.
For $y \in Y$ we write $W_y = W \cap  (\pi^{-1}(y) \times \pi^{-1}(y))$.

By Corollary \ref{so-res} there is a dense $G_\del$ subset $Y_{so}
\subset Y$ such that for each $y \in Y_{so}$ the projection maps
${\bf{p}}_i : W_y \to \pi^{-1}(y)\ (i=1,2)$
are semiopen.
Next recall that the set  $W_{tr}$ of transitive points of $W$ forms a 
dense $G_\del$ subset of $W$. By the Ulam theorem
there is a dense $G_\del$ subset $Y_{tr} \subset Y$ such that for each $y \in Y_{tr}$,
$W_{tr} \cap  (\pi^{-1}(y) \times \pi^{-1}(y))$ is a dense
$G_\del$ subset of $W_y$.
Thus, the set $Y_{so} \cap Y_{tr} \cap Y_{fs}$ is also a dense $G_\del$ subset of $Y$
(see Lemma \ref{RIM-fs}.4 for the definition of $Y_{fs}$). 
Pick a point $y$ in this intersection.

Let $v$ be a minimal idempotent in $E(X,\Ga)$ such that $vy =y$.
By Lemma \ref{Cp-rel}.4 $vX \subset\pi^{-1}(y)$ is a residual subset of $\pi^{-1}(y)$.
As $y \in Y_{so}$ we conclude that the sets $W \cap {\bf{p}_i}^{-1}(vX)\ (i=1,2)$ are
residual subsets of $W_y$ (see Lemma \ref{so-lem}).
Now pick a point $(x,x')$ in the  set
$$
W_{tr} \cap ({\bf{p}_1}^{-1}(vX) \cap {\bf{p}_2}^{-1}(vX))
$$
(which is residual in $W_y$).
We then have $v(x, x') = (vx, vx') = (x, x')$, so that $\ol{G(x,x')}$ is a minimal set.
On the other hand, being a point in $W_{tr}$, we have 
$\ol{G(x,x')} =W$ and our proof is complete.
%
For the last assertion, take $Y$ to be the trivial one point system.
\end{proof}


\br

\section{The structure of tame metric minimal dynamical systems}\label{main-sec}

\begin{defn}
Let us call a minimal dynamical system $(X,\Ga)$ {\em standard} if it has the following structure:
$$
X \overset{\pi}\to Y,
$$
where $(Y, \Ga)$ is a strongly proximal system
and $\pi$ is a RIM extension.
\end{defn}


The next theorem, from \cite{G0}, is just a description of the RIM shadow diagram (\ref{RIM}),
associated to the map $X \to \{*\}$, the trivial one point system. 
The fact that, when $X$ is tame so is $\tilde{X}$, follows from Theorem \ref{tame-inj}.

\begin{thm}\label{RIM-shadow}
Every minimal system $(X,\Ga)$ has a standard extension
$\tilde{X} \overset{\eta}{\to} X$:
\begin{equation}\label{standard}
\xymatrix
{
& \tilde{X} \ar[dr]^\pi  \ar[dl]_\eta &\\
X & & Y,
}
\end{equation}
%
where $(Y, \Ga)$ is a strongly proximal system, $\eta$ is a strongly proximal extension, 
and $\pi$ a RIM extension.
If $(X,\Ga)$ is tame then so is $(\tilde{X},\Ga)$.
\end{thm}

%
%

%

\br

We can now state and prove our main result.

\begin{thm}\label{main}
Let $\Ga$ be any group and \XGa a
metric tame minimal system. 
Then there exists a commutative diagram

\begin{equation}\label{red}
\xymatrix
{
& \tilde{X} \ar[dd]_\pi  \ar[dl]_\eta & X^* \ar[l]_-{\tet^*} \ar[d]^{\iota}\\
X & & Z \ar[d]^\sig\\
& Y & Y^* \ar[l]^\tet
}
\end{equation}
where
\begin{enumerate}
\item
$\tilde{X}$ is metric minimal and tame.
\item
$\eta$ is a strongly proximal extension.
\item
$\pi$ is a RIM and point distal extension, with unique section $y \mapsto \nu_y$.
\item
$(Y, \Gamma)$ is a strongly proximal system.
\item
The maps $\tet$, $\tet^*$ and $\iota$ are almost one-to-one extensions.
\item
$\sig$ is an isometric extension.
\end{enumerate}
Moreover, the extension $\pi^* = \sig \circ \iota$ is a RIM extension
with a unique section $y^* \mapsto \nu^*_{y^*}$. 
For each $y^* \in Y^*$,  $\tet^*(\nu^*_{y^*}) = \nu_{\tet(y^*)}$
 and 
the measure $\iota(\nu^*_{y^*})$ is
the Haar measure $\la_{y^*}$ on the fiber $\sig^{-1}(y^*)$. 
The restriction $\iota :  ((\pi^*)^{-1}(y^*), \nu^*_{y^*})  \to 
(\sig^{-1}(y^*), \la_{y^*})$ 
is a measure theoretical isomorphism.
\end{thm}

%

\begin{proof}

1.\
By Theorem \ref{RIM-shadow} we have the RIM shadow diagram 
(\ref{standard}) with tame $\tilde{X}$.
By Theorem \ref{PD} the RIM extension $\pi$ is point distal. 

\br

2.\
Next apply the structure theorem of point distal extensions, Theorem \ref{PD-rel},
and consider the first two stages of the corresponding AI tower
\begin{equation}\label{two-steps}
\xymatrix
{
& \tilde{X}\ar[dl]_{\eta} \ar[dd]_{\pi}  &  X^* \ar[l]_-{\tet^*}  \ar[dd]_{\pi^*}    \ar[ddr]^\iota  &  
& X^{**}\ar[ll]_-{\tet_1^*} \ar[dd]_{\pi^{**}} \ar[ddr]^{\psi}  &\\
X & & & & & \\
& Y  & Y^* \ar[l]^{\tet} &  Z \ar[l]^\sig & Z^*\ar[l]^{\tet_1} & \hat{X} \ar[l]^\rho}
\end{equation}
Thus 
\begin{itemize}
\item
the diagram
\begin{equation*}
\xymatrix
{
\tilde{X} \ar[d]_{\pi}  &  X^*\ar[l]_-{\tet^*} \ar[d]^{\pi^*}\\
Y  & Y^*\ar[l]^{\tet}
}
\end{equation*}
is the O-shadow diagram associated to $\pi$
(note that we no longer know that the system $X^*$
is tame),  
\item
$\sig$ is the largest intermediate
isometric extension for $\pi^*$, 
\item
the diagram
\begin{equation*}
\xymatrix
{
X^*   \ar[d]_\iota    & X^{**}\ar[l]_{\tet_1^*} \ar[d]^{\pi^{**}} \\
Z & Z^*\ar[l]^{\tet_1}
}
\end{equation*}
is the O-shadow diagram associated to $\iota$, 
\item
and finally, $\rho :\hat{X} \to Z^*$ is the largest intermediate
isometric extension for $\pi^{**}$. 
\end{itemize}

%
%

Note that if the map $\rho$ is an isomorphism; i.e. the largest 
isometric extension corresponding to $\pi^{**}$ is trivial,
then the AI tower for $\pi$ collapses to the diagram (\ref{red})
which is exactly what we are after.
Thus our goal is to show that indeed $\rho$ is necessarily an isomorphism.

\br

%
%

3.\ 
In order to arrive at a contradiction, we now assume that $\rho$ is a nontrivial isometric extension.
The map $\rho$, being isometric, is open RIM extension with unique section 
$z^* \mapsto \la_{z^*}$ (where $\la_{z^*}$ is the Haar measure on the homogeneous
space of the associated compact group, see the diagram (\ref{g-ext})). 

Consider the subdiagram of (\ref{two-steps})
$$
\hat{X} \overset{\rho}{\to} Z^* \overset{\tet_1}{\to} Z  \overset{\sig}{\to} Y^*.
$$
We apply Proposition \ref{diff} (with
$\hat{X}$ in the role of $X$,
$Y^*$ in the role of $Y$, $\tet_1$ replacing $\tet$,  and
$\phi = \sig \circ \tet_1 \circ \rho : \hat{X} \to Y^*$).
Let $y^*_0 \in Y^*$ be the unique point with $uy^*_0=y^*_0$. 
Denote
$\hat{X}_0 = \phi^{-1}(y^*_0)$ and $X_0^{**} = \psi^{-1}(\hat{X}_0)
= (\pi^{*} \circ \tet_1^*)^{-1}(y_0^*) $.

Set $\Mf_m(\hat{X}_0)=\{\La(\nu): \nu \in \Mf(Z_0^*)\}$.
By Proposition \ref{diff}, the set
$$
\hat{R} = \{\nu \in \Mf(\hat{X}_0): \text{the orbit closure of $\nu$ meets $\Mf_m(\hat{X}_0)$}\}
$$
is a dense $G_\del$ subset of $\Mf(\hat{X}_0)$.

%

\br

%
%
%
%
%
%
%
%

Set $\tilde{X}_0 = (\tet^* \circ \tet_1^*)(X_0^{**})$.
Also, set
\begin{gather*}
\Mf_m(X_0^{**}) := \psi^{-1}(\Mf_m(\hat X_0))
\ {\text{and}}\ \\
\Mf_m(\tilde{X}_0) :=(\tet^* \circ \tet_1^*)( \Mf_m(X_0^{**}))=(\tet^* \circ \tet_1^*)(\psi^{-1}(\Mf_m(\hat X_0))).
\end{gather*}
Let
\begin{equation*}\label{R*}
R^{**} :=\psi^{-1}(\hat R) =
\{\xi \in \Mf(X_0^{**}): \text{the orbit closure of $\xi$ meets
$\Mf_m(X_0^{**})$}\}
\end{equation*}
and
$$
R' := (\tet^* \circ \tet_1^*)(R^{**}) \subset R,
$$ 
where
\begin{equation}\label{R}
R =
\{\xi \in \Mf(\tilde{X}_0): \text{the orbit closure of $\xi$ meets
$\Mf_m(\tilde{X}_0)$}\}.
\end{equation}
By Lemma \ref{so-min} the map $\psi: X^{**} \to \hat{X}$ is semiopen
and by Lemma \ref{so-res-lem} there is a dense $G_\del$ subset $Y^*_0 \subset Y^*$
such that $\psi : (\pi^{*} \circ \tet_1^*)^{-1}(y^*) \to \phi^{-1}(y^*)$ is semiopen
for every $y^* \in Y^*_0$. By Lemma \ref{G-min-id}.2 we may and will assume
that $y^*_0 \in Y^*_0$, so that $\psi : X^{**}_0 \to \hat{X}_0$ is semiopen.
Theorem \ref{so-thm} implies that $\psi: \Mf(X^{**}) \to \Mf(\hat{X}_0)$
is semiopen as well.

Now $\hat R$ is a dense subset of $\Mf(\hat{X}_0)$ and Lemma \ref{so-lem}
implies that $R^{**}$ is a dense subset of $\Mf(X_0^{**})$.
Therefore $R'$ is a dense subset of $\Mf(\tilde{X}_0)$.
From the definition of $R$ (\ref{R}) it is easy to deduce that it
is a $G_\del$ set and because it contains $R'$,
it is in fact a dense $G_\del$ subset of $\Mf(\tilde{X}_0)$.

\br

4.\
Recall that the system $(\tilde{X},\Ga)$ is tame
and, by Theorem \ref{tame-inj} so is $(\Mf(\tilde{X}),\Ga)$.
Moreover we have
$E(\tilde{X},\Ga)=E(\mathfrak{M}(\tilde{X}),\Ga)$.
In particular $u\in E(\mathfrak{M}(\tilde{X}),\Ga)$, as a Baire
1 function, has a dense $G_\del$ set of continuity points.
Moreover, $C_u = C_{\ol{u{\mathfrak{M}(\tilde{X})}}}(u)$,
the set of continuity points of $u$ restricted to the set
$\ol{u{\mathfrak{M}(\tilde{X})}}$, is a dense $G_\del$ subset of $\ol{u{\mathfrak{M}(\tilde{X})}}$
and, by Lemma \ref{Cp-rel}.3,
$$
C_u  \subset
 u\mathfrak{M}(\tilde{X}).
 $$
%

As $\rho$ and $\sig$ are isometric extensions and $\tet_1$ is
an almost one-to-one extension,
it is easy to see that $\ol{u\hat{X}} = \hat{X}_0$.
As $\psi$ is semiopen and a point distal extension, it follows that $\ol{uX^{**}} = X_0^{**}$
and therefore also that $\ol{u\tilde{X}} =\tilde{X}_0$.
It follows that the collection of finite convex combinations
of point masses picked from $u\tilde{X}$ forms a dense
subset of $\Mf(\tilde{X}_0)$.
In turn, this implies that $\ol{u\Mf(\tilde{X})} = \Mf(\tilde{X}_0)$
and we conclude that
$S:=C_u \cap R \subset  u\mathfrak{M}(\tilde{X})$ is a dense $G_\del$ 
subset of $\Mf(\tilde{X}_0)$.


Now if $\nu \in S$ then
$u\nu =\nu$ and, $u$ being a minimal idempotent, the
closure of the $\Ga$ orbit of $\nu$ in $\mathfrak{M}(\tilde{X})$
is a minimal set.
On the other hand, $\nu$ being also an element of $R$, 
this orbit closure meets $\Mf_m(\tilde{X}_0)$.
We conclude that $\ol{\Ga \nu}$ is contained in $\Mf_m(\tilde{X})$. 
In particular $\nu \in \Mf_m(\tilde{X}) \cap \Mf(\tilde{X}_0) = \Mf_m(\tilde{X}_0)$
and we conclude that
$S \subset \Mf_m(\tilde{X}_0)$, whence
the equality
\begin{equation}\label{equ}
\Mf_m(\tilde{X}_0)=\Mf(\tilde{X}_0).
\end{equation}

\br

5.\
Given a point $\tilde{x} \in \tilde{X}_0$, the corresponding point mass
$\del_{\tilde{x}}\in \Mf(\tilde{X}_0)$ must have, by (\ref{equ}),
a preimage in $\Mf_m(X^{**}_0)$, say $(\tet^* \circ \tet_1^*)(\xi)=\del_{\tilde{x}}$
with $\xi \in \Mf_m(X^{**}_0)$.
For $\tilde{x}$ with ${(\tet^* \circ \tet_1^*)}^{-1}(\tilde{x})=\{x^{**}\}$ a singleton,
we must have
$\xi=\del_{x^{**}} \in \Mf_m(X^{**}_0)$ and therefore $\psi_*(\del_{x^{**}})=
\del_{\hat x}\in \Mf_m(\hat X_0)$ with $\hat x = \psi(x^{**})$.
By the definition of $\Mf_m(\hat X_0)$ there exists a measure $\zeta
\in \Mf(Z_0^*)$ with
$$
\del_{\hat x} = \La(\zeta) =\int_{Z^*} \la_{z^*}\, d\zeta(z^*).
$$
This clearly implies that the measure $\zeta$ is a point mass,
say $\zeta=\del_{z^*}$ and that the measure $\la_{z^*}$ --- which is the
Haar measure on the homogeneous space which forms the fiber
$\rho^{-1}(z^*) \subset \hat X$ --- is also a degenerate point mass.
That is, the isometric extension $\rho$ is in fact an isomorphism.
As we observed above, the collapse of $\rho$ implies the collapse of the entire
AI tower for $\pi$,
and the diagram (\ref{red}) is obtained.

\br

6.\
By Lemma \ref{ORIM}
the section for $\pi^*$ has the form $y^* \mapsto \nu^*_{y^*} = \nu_{\tet(y^*)} \times \del_{\tet(y^*)}$,
whence $\tet^*(\nu^*_{y^*}) = \nu_{\tet(y^*)}$.
Now the extension $\sig : Z \to Y^*$, being isometric, admits
a unique section $y^* \mapsto \la_{y^*}$, where $\la_{y^*}$
is the Haar measure on the homogeneous space $\sig^{-1}(y^*)$.
It follows that $\iota(\nu^*_{y^*})$ is
the Haar measure $\la_{y^*}$ on the fiber $\sig^{-1}(y^*)$. 

From the construction of the O-shadow diagram it follows that 
the map 
$$
\tet^* : (\pi^*)^{-1}(y^*) \to \pi^{-1}(y)
$$ 
(with $y = \tet(y^*)$)
is a homeomorphism. Therefore we can study the nature of the map
$\iota :  ((\pi^*)^{-1}(y^*), \nu_{y^*})  \to (\sig^{-1}(y^*), \la_{y^*})$ 
via its pushforward 
$$
\iota \circ (\tet^*)^{-1} :  (\pi^{-1}(y), \nu_{y})  \to (\sig^{-1}(y^*), \la_{y^*}).
$$
Let $vy^*=y^*$ for a minimal idempotent  $v$. Then also $vy=y$
and, $\tilde{X}$ being tame, we have $\nu_y(v\tilde{X}) =1$
by Lemma \ref{Cp-rel}.2.
We disintegrate $\nu^*_{y^*}$ over $\la_{y*}$
$$
\nu^*_{y^*} = \int_{\sig^{-1}(y^*)} \xi_z d\, \la_{y*}(z),
$$
and correspondingly
$\nu_y$ over $\la_y$, say
$$
\nu_y = \int \xi_\om d\, \la_y(\om).
$$
Now in the first disintegration, for $\la_{y^*}$ almost every $z$,
the measure $\xi_z$ is concentrated on  $\iota^{-1}(z)$, which
consists of pairwise proximal points ($\iota$ being a proximal extension).
On the other hand, any two points in $v\tilde{X}$ are distal. We therefore
conclude that for $\la_{y^*}$ almost every $z$, the measure $\xi_\om$,
and therefore also $\xi_z$, is a point mass. This means that indeed
$\iota :  ((\pi^*)^{-1}(y^*), \nu_{y^*})  \to (\sig^{-1}(y^*), \la_{y^*})$ 
is a measure theoretical isomorphism.

Since the Haar measures section is unique we conclude that the section $y^* \mapsto \nu^*_{y^*}$
is unique. Finally, by Lemma \ref{ORIM}, also $y \mapsto \nu_y$ is unique.
\end{proof}

%


\noindent\begin{cor}\label{cor}
With notations as in Theorem \ref{main}
\begin{enumerate}
\item
If the RIM extension $\pi$ happens to be open then the diagram (\ref{red}) 
reduces to the simplified form
\begin{equation}\label{cor-diag}
\xymatrix
{
& \tilde{X}  \ar[dl]_\eta  \ar[d]^{\iota}\\
X &  Z \ar[d]^\sig\\
& Y
}
\end{equation}
with $\pi = \sig \circ \iota$.
\item
If the system $(X, \Ga)$ admits an invariant probability measure $\mu$ then 
$X$ is almost automorphic,
i.e. it has the form $X \overset{\iota}{\to} Z$, where $Z$ is equicontinuous and $\iota$
is almost one-to-one. Moreover, $\mu$ is unique and the map $\iota$ is a measure theoretical isomorphism $\iota : (X, \mu, \Ga) \to (Z, \la, \Ga)$, where  $\la$ is the Haar measure on the 
homogeneous space $Z$.
\item
Thus, when $\Ga$ is amenable this latter situation is the rule.
\end{enumerate}
\end{cor}

Concerning the first assertion in Corollary \ref{cor} we pose the following.

\begin{prob}
Is a RIM homomorphism $\pi : X \to Y$
of minimal metric systems necessarily open ?
Is this true when, in addition, we assume that $X$ is tame ?
\end{prob}

\begin{rmk}
%
In the situation described in Corollary \ref{cor}.2,
the set $X_0= \{x\in X: \iota^{-1}(\iota(x)) = \{x\}\}$
is a dense $G_\del$ and $\Ga$-invariant subset
of $X$ and thus has $\mu$ measure either zero or one.
In \cite[Section 11]{KL} Kerr and Li construct a minimal Toeplitz
cascade (i.e. a $\Z$-system) which is tame and not null.
In \cite[Remark 5.2]{Gmt} I claimed that this example 
can be made not regular in the sense that the densities of
the periodic parts converge to $d < 1$. As for such non-regular systems
$\mu(X_0)=0$, this would show that the unique invariant measure of a minimal tame system need not be supported by the set $X_0$ where 
$\pi$ is 1-1. Unfortunately the argument I had in mind when claiming that the Kerr-Li example can be made nonregular was flawed.
Thus the following basic question is still open.
\end{rmk}

\begin{prob}
Let $(X, \Ga)$ be a metric minimal tame system.
Suppose that $(X, \Ga)$ admits
a (necessarily unique) invariant probability measure $\mu$.
Let $X \overset{\iota}{\to} Z$ be its maximal equicontinuous factor 
(so that  $\iota$ is an almost one-to-one extension). 
Let $X_0 \subset X$ be the dense $G_\del$ subset of $X$
defined by
$X_0= \{x\in X: \iota^{-1}(\iota(x)) = \{x\}\}$.
Is it necessarily the case that $\mu(X_0)=1$ ?
\end{prob}

%

\begin{rmk}
On the other hand, one may ask whether every
metric minimal almost automorphic system $(X, \Ga)$,
with $X \overset{\iota}{\to} Z$ being its maximal equicontinuous factor 
(so that  $\iota$ is an almost one-to-one extension), which admits a unique invariant
measure $\mu$ such that $\mu(X_0)=1$, 
is necessarily tame ?
The answer here is no. We can construct a minimal subshift $X
\subset \{0, 1\}^\Z$ as above, with $Z$ being an irrational rotation of the circle,
yet $(X,\sig)$ ($\sig$ denoting the shift) is not tame. We omit the details which will be published elsewhere.
\end{rmk}

\begin{rmk}
The following corollary of Theorem \ref{main} was suggested by the referee.
The notion of mean equicontinuity was studied in \cite{LTY} and \cite{DG}.
(A $\Z$-system is {\em mean equicontinuous} if for every 
$\ep >0$ there is $\del > 0$ such that if 
$d(x,y) < \del$ then $d(T^ix,T^iy) < \ep$ for all $i$ except for a set of density  $< \ep$.)
%
Combining Theorem 2.1 of \cite{DG} and  Theorem \ref{main} we deduce the following:

\begin{cor}
A minimal tame $\Z$-system is mean equicontinuous. 
\end{cor}

\end{rmk}

\br

\section{Examples}
We illustrate our main theorem, Theorem \ref{main}, with some basic examples.
\begin{itemize}
\item
The classical Sturmian $\Z$ dynamical system $(X,T)$ 
(where $T : X \to X$ is the homeomorphism
which generate the $\Z$ action) 
is minimal and tame. This is an almost automorphic system
with $\tet : X \to Z$, an almost one-to-one factor, the projection of $X$ onto its 
largest equicontinuous factor $Z$,
the dyadic adding machine. See \cite{GM}.

\item 
Every {\bf null} dynamical system is tame; see \cite{HLSY} and \cite{KL}.
(A system is null if it has zero sequence topological entropy
with respect to every subsequence $n_i \nearrow \infty$.)

\item
The action of $G=GL(d,\R)$ (and hence of any
subsgroup $\Ga < G$) on the projective space $\mathbb{P}^{d-1}$,
comprising the lines through the origin in $\R^n$, is tame
(see the appendix below).
The same is true for the action of $G$ on $S^{d-1}$,
identified with the space of rays emanating from the origin.
We have a natural two-to-one map $\sig : S^{d-1} \to \mathbb{P}^{d-1}$.
This factor map corresponds to the diagram (\ref{cor-diag}) with
$X = \tilde{X} = Z = S^d$ and $Y = \mathbb{P}^{d-1}$. See \cite{F-B}, \cite{Ellis93} and \cite{Ak-98}.
\item
The action of the group $G = \Homeo(S^1)$ (and hence of any
subsgroup $\Ga < G$) on the circle $S^1$ is tame. See \cite{GM16}.

\item

Our next example will demonstrate the necessity for the presence of the extension $\eta$
in the diagram (\ref{cor-diag}). It is an elaboration of an example of Furstenberg and Glasner,
see \cite{G00}.
Let $G$ be the closed subgroup of the Lie group
$GL(4,\R)$ consisting of all $4 \times 4$ matrices of the form
\begin{equation*}
\begin{pmatrix}
A & 0 \\
0 & B
\endpmatrix
\qquad\text{and} \qquad
\pmatrix
0 & A \\
B & 0,
\end{pmatrix}
\end{equation*}
with $A,B\in GL(2,\R)$. We let $G$ act on the subspace $X$ of the
projective space $\mathbb{P}^3$ consisting of the disjoint union of the
two one dimensional projective spaces $\mathbb{P}^1$, which are
naturally embedded in $\mathbb{P}^3$, the quotient space of
$\R^4=\R^2\times \R^2$. 
Call these two copies $X_1$ and $X_2$ respectively.  
By \cite{Ak-98} this action is tame.

There is a natural projection from 
$(X,G)$ 
onto the
two points $G$-system $(\{X_1,X_2\},G)$. 
It is now easy to establish the remaining assertions of the following:

{\bf Claim 1}: The $G$-system $(X,G)$
is minimal and tame. It admits the isometric factor
which is the ``flip" on two points and the map from $X$ onto $(\{X_1,X_2\},G)$
is a strongly proximal extension. However, the system $(X,G)$ admits no
nontrivial proximal factor.

Let $Y$ be the ``quasifactor" of $(X,G)$ defined by:
$$
Y=\{\frac12(\del_{x_1}+ \del_{x_2}) : x_i \in X_i, \  i=1,2\} \subset \Mf(X).
$$
Again it is easy to check that the system $(Y,G)$ is 
minimal and strongly proximal $G$-system.  

Next let
$$
\tilde{X} = \{ (x_i,\frac12(\del_{x_1}+ \del_{x_2})) :  x_i \in X_i,\   i=1,2\}
\subset X \times \Mf(X). 
$$



{\bf Claim 2:} 
$\tilde{X}$ is minimal and tame.
Let $\eta : \tilde{X} \to X$ and $\sig : \tilde{X} \to Y$
denote the projections from $\hat{X}$ onto its two components. 
Then the diagram (\ref{cor-diag}) (with $Z = \tilde{X}$ and $\iota$ the identity map) is the canonical
standard extension describing the structure of the minimal tame system $X$. 
\end{itemize}

\br

\section{Appendix: Borel's density theorem}

In this short section we will rewrite Furstenberg's proof of Borel's density theorem \cite{F-B},
in terms of tame systems and enveloping semigroups.

As we will show next the action of $G=SL_d(\R)$ on 
the projective space $\mathbb{P}^{d-1}$ ($d \ge 2$) is tame. 
This is also true for the $G$ action on $S^{d-1}$.
In both cases the
enveloping semigroup is a Rosenthal compactum but not metrizable.
In the case of the projective space it is not even first countable (see \cite{Ak-98}).
Thus, these systems are tame but not HAE.


%
%
%

For a non-zero vector $v \in \R^d$ we let $\overline{v}$
 denote its image in $\mathbb{P}^{d-1}$.
Similarly, for $g \in G$ we let $\overline{g}$ denote its image in $E(\mathbb{P}^{d-1},SL_d(\R))$.

\begin{lem}\label{Flem}
The action of $GL_d(\R)$ on $\mathbb{P}^{d-1}$ is tame. 
\end{lem}

\begin{proof}
Let $g_n$ be a sequence of matrices in $GL_d(\R)$.
Given any subspace $W \subset \R^d$, by
passing to a subsequence and by choosing appropriate scalars $\la_n$, we can assume that
$\la_n g_n \to h_W$, a non-zero linear map from $W$ into $\R^d$.
For $v \not\in \ker(h_W)$ we have $\ol{\la_ng_n(v)} \to \ol{h_W(v)}$.

Now define $W_0 =\R^d$, $W_1 = \ker(h_{W_0})$, and repeat this procedure with
$W_1$ to obtain $h_{W_1}$ which maps $W_1$ onto a subspace of $\R^d$. 
Proceeding by induction (till we have a trivial kernel) we finally obtain a subsequence
such that $\ol{g_{n_i} }\to p \in E(\mathbb{P}^{d-1},SL_d(\R))$.
This shows that $(\mathbb{P}^{d-1}, G)$ is tame. 
\end{proof}

Note that for $W =\R^d$ the assumptions $g_n \in SL_d(\R)$ and $\|g_n\| \to \infty$, imply that $\dim \ker h_W \ge 1$. 
Thus, in this case the image of $p$ is
a finite union of projective sub-verieties of dimension $< d-1$ ({\em a quasi-sub-veriety}).



\begin{defn}
 A pair of groups $(G, H)$ is called a {\em Borel pair} if $G$ is a minimally almost periodic (m.a.p.)
group and $H$ is a closed subgroup of $G$ such that the quotient space $G/H$ supports a finite 
 $G$-invariant measure.
\end{defn}

 \begin{thm}[Borel's density theorem]
 Let $(G, H)$ be a Borel pair and $\pi$ a finite-dimensional representation of $G$ on a space $V$. 
 If $W$ is a $\pi(H)$-invariant subspace, it is also $\pi(G)$-invariant. 
 \end{thm}

\begin{proof}
Suppose first that $\dim(W)=1$, say $W = \R v$, so that $x_0 = \ol{v} \in \mathbb{P}(V)$. 
Then, the map $g \mapsto \ol{gv}= gx_0$,
from $G/H$ into $\mathbb{P}(V)$, sends the invariant measure on $G/H$ to an invariant measure,
say $\mu$, on $\mathbb{P}(V)$. Let $Y$ be the support of $\mu$.
Clearly $x_0 \in Y$ 
and $Y$ is a closed $G$-invariant subset.
Let $u \in E(\mathbb{P}(V), G)$ be a minimal idempotent. Then (by tameness; see Theorem \ref{tame-inj})
$u_*\mu = \mu$ and it follows
 that $\mu(u\mathbb{P}(V))=1$.
By tameness there is a sequence $g_n \in G$ with $\ol{g_n} \to u$.
The range of $u$ can not be all of $\mathbb{P}(V)$ since this would mean that
the enveloping semigroup is actually a group and therefore, that 
the action is distal; however, this possibility is ruled out by the m.a.p. property of $G$.
Thus, the set $u\mathbb{P}(V)$ is a quasi-sub-veriety
(finite union of proper sub-varieties). 
In particular it is closed, whence $Y \subset  u\mathbb{P}(V)$.

Let $L = \bigcup L_i$ be a minimal projective quasi-sub-variety containing $Y$.
Then $L$ is closed and $G$-invariant and $G$ must permute its components. By m.a.p.
the permutations are all the identity permutations. Thus each $L_i$ is $G$-invariant.

If  the $G$-action on $L_i$ is not trivial then, again we conclude that $uL_i \subsetneq L_i$,
which contradicts the minimality of $L$. Thus $G$ acts on $L$ as the identity and, in particular,
$x_0 \in Y \subset L$ is a $G$-fixed point. 
%
%
%
%

The general case is reduced to the $1$-dimensional one using exterior products. 
\end{proof}

We refer the reader to \cite{Bo} and \cite{F-B} where it is shown how from this theorem one 
easily deduces the remaining results of \cite{Bo}.


\bibliographystyle{amsplain}

\end{document}